\documentclass[12pt]{amsart}
\usepackage{amssymb,amsmath,hyperref}
\usepackage[all]{xy}

\newcommand{\lra}{\longrightarrow}
\newcommand{\ra}{\rightarrow}

\newcommand{\bC}{\mathbb C}
\newcommand{\bD}{\mathbb D}
\newcommand{\bF}{\mathbb F}

\newcommand{\bP}{\mathbb P}
\newcommand{\bQ}{\mathbb Q}
\newcommand{\bR}{\mathbb R}
\newcommand{\bZ}{\mathbb Z}

\newcommand{\cA}{\mathcal A}
\newcommand{\cC}{\mathcal C}
\newcommand{\cD}{\mathcal D}
\newcommand{\cI}{\mathcal I}
\newcommand{\cN}{\mathcal N}
\newcommand{\cO}{\mathcal O}

\newcommand{\Aut}{\operatorname{Aut}}
\newcommand{\Bl}{\operatorname{Bl}}
\newcommand{\CH}{\operatorname{CH}}
\newcommand{\Gr}{\operatorname{Gr}}
\newcommand{\Pf}{\operatorname{Pf}}
\newcommand{\Pfaff}{\operatorname{Pfaff}}
\newcommand{\Pic}{\operatorname{Pic}}
\newcommand{\prim}{\operatorname{prim}}
\newcommand{\Sym}{\operatorname{Sym}}
\newcommand{\Scroll}{\operatorname{Scr}}

\theoremstyle{plain}
\newtheorem{prop}{Proposition}
\newtheorem{theo}[prop]{Theorem}
\newtheorem{coro}[prop]{Corollary}
\newtheorem{lemm}[prop]{Lemma}
\theoremstyle{remark}
\newtheorem{rema}[prop]{Remark}
\theoremstyle{definition}
\newtheorem{defi}[prop]{Definition}
\newtheorem{exam}[prop]{Example}
\newtheorem{ques}[prop]{Question}

\numberwithin{equation}{section}

\author{Brendan Hassett}
\address{Department of Mathematics \\
Brown University \\
Box 1917 \\
151 Thayer Street \\
Providence, RI 02912 \\USA}
\email{bhassett@math.brown.edu}

\title[Cubic fourfolds, K3 surfaces, and rationality questions]{Cubic fourfolds, K3 surfaces, and rationality questions}
\begin{document}

\maketitle


This is a survey of the geometry of complex cubic
fourfolds with a view toward rationality questions. 
Smooth cubic surfaces have been known to be rational
since the 19th century \cite{DolCr}; cubic threefolds
are irrational by the work of Clemens and Griffiths
\cite{CG}. Cubic fourfolds are likely more varied
in their behavior.
While there are examples known
to be rational, we expect that most cubic fourfolds
should be irrational. However, no cubic
fourfolds are {\em proven} to be irrational.

Our organizing principle is that
progress is likely to be driven
by the dialectic between concrete geometric constructions
(of rational, stably rational,
and unirational parametrizations) 
and conceptual tools differentiating various classes
of cubic fourfolds
(Hodge theory, moduli spaces and derived
categories, and decompositions of the diagonal). 
Thus the first section of this paper is devoted to 
classical examples of rational parametrizations.
In section two we focus on Hodge-theoretic 
classifications of cubic fourfolds with various 
special geometric structures. 
These are explained in section three using techniques
from moduli theory informed by deep results on
K3 surfaces and their derived categories.
We return to constructions in the
fourth section, focusing on unirational 
parametrizations of special classes of cubic fourfolds.
In the last section, we touch on recent 
applications of decompositions of the diagonal
to rationality questions, and what they mean for
cubic fourfolds.

\subsection*{Acknowledgments} This work is partially supported by NSF grant
1551516. 
I am grateful to N.~Addington, A.~Auel, H.~Nuer,
R.~Thomas, 
Y.~Tschinkel, and C.~Voisin for comments on this survey
and 
J.-L.~Colliot-Th\'el\`ene, J.~Harris, B.~Mazur, and
A.~V\'arilly-Alvarado
for discussions informing its content.
I would also like to thank the anonymous referee for pointing out work of
Crauder and Katz that invalidated the original formulation
of Question~\ref{ques:Cremona}.
I am grateful to CIME and CIRM 
for organizing the school `Rationality problems
in algebraic geometry' where this material was to be presented,
and the organizers Rita Pardini and Gian Pietro Pirola for
supporting this survey despite my inability to lecture at the school.

\section{Introduction and classical constructions}

Throughout, we work over the complex numbers. 

\subsection{Basic definitions}
\label{ss:basic}

Let $X$ be a smooth projective variety of dimension $n$.
If there exists
a birational map
$\rho: \bP^n \stackrel{\sim}{\dashrightarrow} X$
we say that $X$ is {\em rational}. It is {\em stably rational}
if $X\times \bP^m$ is rational for some $m\ge 0$. 
If there exists 
a generically finite map
$\rho:\bP^n \dashrightarrow X$
we say that $X$ is {\em unirational}; this is equivalent
to the existence of a dominant map from a projective
space of arbitrary dimension to $X$.

A {\em cubic fourfold}
is a smooth cubic hypersurface $X \subset \bP^5$, with defining
equation 
$$F(u,v,w,x,y,z)=0$$
where $F\in \bC[u,v,w,x,y,z]$ is homogeneous of degree three.
Cubic hypersurfaces in $\bP^5$ are parametrized by
$$\bP(\bC[u,v,w,x,y,z]_3) \simeq \bP^{55}$$
with the smooth ones corresponding to a Zariski open 
$U \subset \bP^{55}$.

Sometimes we will consider singular cubic hypersurfaces; in
these cases, we shall make explicit reference to the singularities.
The singular cubic hypersurfaces in $\bP^5$ are parametrized by an
irreducible divisor
$$\Delta:=\bP^{55} \setminus U.$$
Birationally, $\Delta$ is a $\bP^{49}$ bundle over $\bP^5$,
as having a singularity at a point $p\in \bP^5$ imposes six
independent conditions.

The moduli space of cubic fourfolds is the quotient
$$\cC:=[U/\operatorname{PGL}_6].$$
This is a Deligne-Mumford stack with quasi-projective coarse
moduli space, e.g., by classical results on the
automorphisms and invariants of hypersufaces \cite[ch.~4.2]{MFK}. 
Thus we have
$$\dim(\cC)=\dim(U)-\dim(\operatorname{PGL}_6)=55-35=20.$$

\subsection{Cubic fourfolds containing two planes}
\label{subsect:twoplanes}
Fix disjoint projective planes 
$$P_1=\{u=v=w=0\},  P_2=\{x=y=z=0\} \subset \bP^5$$
and consider the cubic fourfolds $X$ containing $P_1$
and $P_2$. For a concrete equation, consider
$$X=\{ux^2+vy^2+wz^2=u^2x+v^2y+w^2z\}$$
which is in fact smooth! 
See \cite[\S5]{HuKl} for
more discussion of this example.

More generally, fix forms
$$F_1,F_2 \in \bC[u,v,w;x,y,z]$$
of bidegree $(1,2)$ and $(2,1)$ in the variables
$\{u,v,w\}$ and $\{x,y,z\}$. Then the cubic hypersurface
$$X=\{F_1+F_2=0\} \subset \bP^5$$
contains $P_1$ and $P_2$, and the defining equation of every
such hypersurface takes that form.
Up to scaling, these form a projective space of dimension $35$. 
The group
$$\{g \in \operatorname{PGL}_6:g(P_1)=P_1,g(P_2)=P_2 \}$$
has dimension $17$. Thus the locus of cubic fourfolds containing a 
pair of disjoint planes has codimension two in $\cC$.

The cubic fourfolds of this type are {\em rational}.
Indeed, we construct a birational map as follows:
Given points $p_1\in P_1$ and $p_2 \in P_2$, let
$\ell(p_1,p_2)$ be the line containing them. The Bezout
Theorem allows us to write
$$\ell(p_1,p_2)\cap X =\begin{cases} \{p_1,p_2,\rho(p_1,p_2)\}
			\text{ if } \ell(p_1,p_2) \not \subset X \\
			 \ell(p_1,p_2) \text{ otherwise.}
		\end{cases}
$$
The condition $\ell(p_1,p_2) \subset X$ is expressed by the equations
$$S:=\{F_1(u,v,w;x,y,z)=F_2(u,v,w;x,y,z)=0\} \subset
P_{1,[x,y,z]}\times P_{2,[u,v,w]}.$$
Since $S$ is a complete intersection of hypersurfaces of bidegrees $(1,2)$
and $(2,1)$ it is a K3 surface, typically with Picard group of rank two.
Thus we have a well- defined morphism
$$\begin{array}{rcl}
\rho:P_1 \times P_2  \setminus S & \ra & X \\
      (p_1,p_2) & \mapsto & \rho(p_1,p_2)
\end{array}
$$
that is birational, as each point of $\bP^5 \setminus (P_1\cup P_2)$ 
lies on a unique line joining the planes.  

We record the linear series inducing this birational parametrization:
$\rho$ is given by the forms of bidegree $(2,2)$ containing $S$ and
$\rho^{-1}$ by the quadrics in $\bP^5$ containing $P_1$ and $P_2$;

\subsection{Cubic fourfolds containing a plane and odd multisections}
\label{subsect:POM}
Let $X$ be a cubic fourfold containing a plane $P$. Projection from
$P$ gives a quadric surface fibration
$$q:\tilde{X}:=\Bl_P(X) \ra \bP^2$$
with singular fibers over a sextic curve $B\subset \bP^2$.
If $q$ admits a rational section then $\tilde{X}$ is rational
over $K=\bC(\bP^2)$ and thus over $\bC$ as well.
The simplest example of such a section is another plane
disjoint from $P$.
Another example was found by Tregub \cite{Tregub2}:
Suppose there is a quartic Veronese surface
$$V \simeq \bP^2 \subset X$$
meeting $P$ transversally at three points. Then
its proper transform $\tilde{V} \subset \tilde{X}$ is a section
of $q$, giving rationality.

To generize this, we employ a basic property of quadric surfaces 
due to Springer (cf.~\cite[Prop.~2.1]{HaJAG} and \cite{Swan}):
\begin{quote}
Let $Q\subset \bP^3_K$ be a quadric surface smooth over a field
$K$. Suppose there exists an extension $L/K$ of odd degree
such that $Q(L)\neq \emptyset$. Then $Q(K)\neq \emptyset$
and $Q$ is rational over $K$ via projection from a rational point.
\end{quote}
This applies when there exists a surface $W\subset X$ intersecting the
generic fiber of $q$ transversally in an odd number of points.
Thus we the following:
\begin{theo} \label{theo:JAG} \cite{HaJAG}
Let $X$ be a cubic fourfold containing a plane $P$ and 
projective surface $W$ such that
$$\deg(W)-\left<P, W\right>$$
is odd. Then $X$ is rational.
\end{theo}
The intersection form on the middle cohomology of $X$ is denoted by
$\left<,\right>$.

Theorem~\ref{theo:JAG} 
gives a countably infinite collection of codimension two subvarieties
in $\cC$ parametrizing rational cubic fourfolds.
Explicit birational maps $\rho:\bP^2 \times \bP^2 \stackrel{\sim}{\dashrightarrow}
X$ can be found in many cases \cite[\S5]{HaJAG}.

We elaborate the geometry behind Theorem~\ref{theo:JAG}:
Consider the relative variety of lines of the quadric surface fibration $q$
$$f:F_1(\tilde{X}/\bP^2) \ra \bP^2.$$
For each $p\in \bP^2$, $f^{-1}(p)$ parametrizes the lines contained
in the quadric surface $q^{-1}(p)$. When the fiber is smooth, this is a disjoint
union of two smooth $\bP^1$'s; for $p\in B$, we have a single
$\bP^1$ with multiplicity two. Thus the Stein factorization 
$$f:F_1(\tilde{X}/\bP^2) \ra S \ra \bP^2$$
yields a degree two K3 surface---the double cover $S\ra \bP^2$ branched
over $B$---and a $\bP^1$-bundle $r:F_1(\tilde{X}/\bP^2)\ra S$. 
The key to the proof is the equivalence of
the following conditions (see also \cite[Th.~4.11]{Kuz15}):
\begin{itemize}
\item{the generic fiber of $q$ is rational over $K$;}
\item{$q$ admits a rational section;}
\item{$r$ admits a rational section.}
\end{itemize}
The resulting birational map $\rho^{-1}:X\dashrightarrow \bP^2\times \bP^2$
blows down the lines incident to the section of $q$, which are parametrized
by a surface birational to $S$.

Cubic fourfolds containing a plane
have been re-examined recently from the perspective of twisted
K3 surfaces and their derived categories \cite{Kuz08,MacSte,Kuz15}.
The twisted K3 surface is the pair $(S,\eta)$, where $\eta$ 
is the class in the Brauer group of $S$ arising from $r$;
note that $\eta=0$ if and only if the three equivalent
conditions above hold.
Applications of this geometry to rational points
may be found in \cite{HVAV}.

\begin{rema} 
The technique of Theorem~\ref{theo:JAG} applies to all smooth
projective fourfolds
admiting quadric surface fibrations $Y\ra P$ over a rational surface
$P$.
Having an odd multisection suffices to give rationality.
\end{rema}

\subsection{Cubic fourfolds containing quartic scrolls}
A {\em quartic scroll} is a smooth 
rational ruled surface $\Sigma \hookrightarrow \bP^5$ with
degree four, with the rulings embedded as lines.
There are two possibilities:
\begin{itemize} 
\item{$\bP^1 \times \bP^1$ embedded via the linear series
$|\cO_{\bP^1\times \bP^1}(1,2)|$}
\item{the Hirzebruch surface $\bF_2$ embedded via 
$|\cO_{\bF_2}(\xi+f)|$ where $f$ is a fiber and $\xi$ a
section at infinity ($f\xi=1$ and $\xi^2=2$).}
\end{itemize}
The second case is a specialization of the first.
Note that all scrolls of the first class are projectively equivalent
and have equations given by the $2\times 2$ minors of:
$$\left( \begin{matrix}
u & v & x & y \\
v & w & y & z 
\end{matrix}
\right) $$
\begin{lemm} \label{lemm:scroll}
Let $\Sigma$ be a quartic scroll, realized as the image of 
$\bP^1 \times \bP^1$ under the linear series $|\cO_{\bP^1\times \bP^1}(1,2)|$.
Then a generic point $p\in \bP^5$ lies on a unique secant to $\Sigma$. 
The locus of points on more than one secant equals the Segre threefold
$\bP^1 \times \bP^2$ associated with the Veronese embedding 
$\bP^1 \hookrightarrow \bP^2$ of the second factor.
\end{lemm}
\begin{proof}
The first assertion follows from a computation with the double point
formula \cite[\S9.3]{Fulton}. For the second, if two secants to $\Sigma$,
$\ell(s_1,s_2)$ and $\ell(s_3,s_4)$, intersect then $s_1,\ldots,s_4$ are
coplanar.
But points $s_1,\ldots,s_4 \in \Sigma$ that fail to impose independent
conditions on 
$|\cO_{\bP^1\times \bP^1}(1,2)|$ necessarily have at least
three points on a line or all the points on a conic contained in $\Sigma$.
\end{proof}

Surfaces in $\bP^5$ with `one apparent double point'
have been studied for a long time. 
See \cite{edge} for discussion and additional
classical references and \cite{BRS} for a modern application
to cubic fourfolds.

\begin{prop}
If $X$ is a cubic fourfold containing a quartic scroll $\Sigma$
then $X$ is rational.
\end{prop}
Here is the idea:
Consider the linear series of quadrics cutting out $\Sigma$. It induces
a morphism
$$\Bl_{\Sigma}(X) \ra \bP(\Gamma(\cI_{\Sigma}(2)))\simeq \bP^5,$$
mapping $X$ birationally to a quadric hypersurace. 
Thus $X$ is rational.

\begin{rema} \label{rema:scroll}
Here is another approach. Let $R\simeq \bP^1$ denote the ruling of $\Sigma$;
for $r\in R$, let $\ell(r) \subset \Sigma \subset X$ denote the corresponding
line.
For distinct $r_1,r_2 \in R$, the intersection
$$\operatorname{span}(\ell(r_1),\ell(r_2)) \cap X$$
is a cubic surface containing disjoint lines. Let $Y$ denote
the closure 
$$\{(x,r_1,r_2):x\in \operatorname{span}(\ell(r_1),\ell(r_2)) \cap X \}
\subset X\times \Sym^2(R)\simeq X \times \bP^2.$$
The induced $\pi_2:Y\ra \bP^2$ is a cubic surface fibration such that the generic
fiber contains two lines.  Thus the generic fiber
$Y_K$, $K=\bC(\bP^2)$, is rational over $K$
and consequently $Y$ is rational over $\bC$.   

The degree of $\pi_1:Y \ra X$ can be computed as follows: It is the number
of secants to $\Sigma$ through a generic point $p\in X$. There is
one such secant by Lemma~\ref{lemm:scroll}.
We will return to this in \S\ref{sect:unirat}.
\end{rema}

Consider the nested Hilbert scheme
$$\Scroll=\{\Sigma \subset X \subset \bP^5: \Sigma
\text{ quartic scroll}, X \text{ cubic fourfold} \}$$
and let $\pi:\Scroll \ra \bP^{55}$ denote the morphism forgetting
$\Sigma$. We have $\dim(\Scroll)=56$ so the fibers of $\pi$ are
positive dimensional.
In 1940,
Morin \cite{Morin} asserted that the generic fiber of $\pi$ is
one dimensional, deducing (incorrectly!) that the generic cubic
fourfold is rational. Fano \cite{Fano} corrected this a few years later,
showing that the generic fiber has dimension two; cubic fourfolds
containing a quartic scroll thus form a {\em divisor} in $\cC$.
We will develop a conceptual approach to this in \S\ref{subsect:SCD}.

\subsection{Cubic fourfolds containing a quintic del Pezzo surface}
Let $T\subset \bP^5$ denote a quintic del Pezzo surface, i.e.,
$T=\Bl_{p_1,p_2,p_3,p_4}(\bP^2)$ anti-canonically embedded.
Its defining equations are quadrics
$$Q_i=a_{jk}a_{lm}-a_{jl}a_{km}+a_{jm}a_{kl}, 
\{1,\ldots,5\}=\{i,j,k,l,m\}, j<k<l<m,$$
where the $a_{rs}$ are generic linear forms on $\bP^5$. The rational map
$$\begin{array}{rcl}
Q: \bP^5 & \dashrightarrow & \bP^4 \\
 \  [u,v,w,x,y,z] & \mapsto & [Q_1,Q_2,Q_3,Q_4,Q_5]
\end{array}
$$
contracts each secant of $T$ to a point. Note that a generic $p\in \bP^5$
lies on a unique such secant.

\begin{prop}
A cubic fourfold containing a quintic del Pezzo surface is rational.
\end{prop}
Indeed, restricting $Q$ to 
$X$ yields a {\em birational} morphism
$\Bl_T(X) \ra \bP^4$.

\subsection{Pfaffian cubic fourfolds}
\label{subsect:Pfaff}
Recall that if $M=(m_{ij})$ is skew-symmetric
$2n\times 2n$ matrix then the determinant
$$\det(M)=\Pf(M)^2,$$
where $\Pf(M)$ is a homogeneous form of degree $n$ in the entries
of $M$, known as its Pfaffian. 
If the entries of $M$ are taken as linear forms in $u,v,w,x,y,z$,
the resulting hypersurface 
$$X=\{\Pf(M)=0 \} \subset \bP^5$$ 
is a Pfaffian cubic fourfold.

We put this on a more systematic footing. Let $V$ denote a six-dimensional
vector space and consider the strata
$$\Gr(2,V) \subset \Pfaff(V) \subset  \bP(\bigwedge^2 V),$$
where $\Pfaff(V)$ parametrizes the rank-four tensors. 
Note that $\Pfaff(V)$ coincides with the secant variety to
$\Gr(2,V)$, which is degenerate, i.e., smaller than the expected dimension.
We also have dual picture
$$\Gr(2,V^*) \subset \Pfaff(V^*) \subset  \bP(\bigwedge^2 V^*).$$
A codimension six subspace $L \subset \bP(\bigwedge^2 V)$
corresponds to a codimension nine subspace $L^{\perp} \subset 
\bP(\bigwedge^2 V^*)$. 
Let $X=L^{\perp}\cap \Pfaff(V^*)$ denote the resulting Pfaffian
cubic fourfold and $S=L\cap \Gr(2,V)$ the associated degree fourteen
K3 surface.

Beauville and Donagi \cite{BD} (see also \cite{Tregub1})
established the following properties,
when $L$ is generically chosen:
\begin{itemize}
\item{$X$ is rational: For each codimension one subspace $W \subset  V$,
the mapping 
$$\begin{array}{rcl}
Q_W:X & \dashrightarrow & W \\
\ [\phi] & \mapsto & \ker(\phi)\cap W
\end{array}$$
is birational.
Here we interpret $\phi:V \ra V^*$ as an
antisymmetric linear transformation.}
\item{$X$ contains quartic scrolls: For each point $[P] \in S$, consider
$$\Sigma_{P}:=\{[\phi] \in X: \ker(\phi) \cap P \neq 0 \}.$$
We interpret
$P\subset V$ as a linear
subspace. This is the two-parameter family described by Fano.}
\item{$X$ contains quintic del Pezzo surfaces: For each $W$, consider
$$T_W:=\{[\phi] \in X: \ker(\phi) \subset W \},$$
the indeterminacy of $Q_W$. This is a five-parameter family.}
\item{The variety $F_1(X)$ of lines on $X$ is isomorphic to $S^{[2]}$, the Hilbert scheme of length two
subschemes on $S$.}
\end{itemize}
Tregub \cite{Tregub2} observed the connection between containing a quartic
scroll and containing a quintic del Pezzo surface.
For the equivalence between containing a quintic del Pezzo
surface and the Pfaffian condition, see
\cite[Prop.~9.2(a)]{BeauDet}.
  
\begin{rema}
Cubic fourfolds $X$ containing disjoint planes $P_1$ and $P_2$ admit
`degenerate' quartic scrolls and are therefore limits of Pfaffian
cubic fourfolds \cite{Tregub2}.
As we saw in \S\ref{subsect:twoplanes}, the lines connecting $P_1$ and $P_2$
and contained in $X$ are parametrized by a K3 surface
$$S\subset P_1 \times P_2.$$
Given $s\in S$ generic, let $\ell_s$ denote the corresponding line and
$L_i=\operatorname{span}(P_i,\ell_s)\simeq \bP^3$.
The intersection
$$L_i \cap X=P_i \cup Q_i$$
where $Q_i$ is a quadric surface. The surfaces $Q_1$ and $Q_2$ meet
along the common ruling $\ell_s$,
hence $Q_1 \cup_{\ell_s} Q_2$ is a limit of quartic
scrolls. 
\end{rema}

\begin{rema}[Limits of Pfaffians] \label{rema:limitPfaff}
A number of recent papers have explored smooth limits of Pfaffian cubic fourfolds
more systematically. For analysis of the intersection between
cubic fourfolds containing a plane and limits of the Pfaffian locus,
see \cite{ABBVA}.
Auel and Bolognese-Russo-Staglian\`o \cite{BRS}
have shown that smooth limits of Pfaffian cubic
fourfolds are always rational; \cite{BRS} includes a 
careful analysis of the topology of the Pfaffian locus
in moduli.
\end{rema}

\subsection{General geometric properties of cubic hypersurfaces}
\label{subsect:GGP}
 
Let $\Gr(2,n+1)$ denote the Grassmannian
of lines in $\bP^n$. We have a tautological exact sequence
$$0 \ra S \ra \cO_{\Gr(2,n+1)}^{n+1} \ra Q \ra 0$$
where $S$ and $Q$ are the tautological sub- and quotient bundles,
of ranks $2$ and $n-1$. 
For a hypersurface $X\subset \bP^n$, the variety of lines 
$F_1(X) \subset \Gr(2,n+1)$ parametrizes lines contained in $X$.
If $X=\{G=0\}$ for some homogeneous form $G$ of degree $d=\deg(X)$ then 
$F_1(X)=\{\sigma_G=0\}$, where 
$$\sigma_G \in \Gamma(\Gr(2,n+1),\Sym^d(S^*))$$
is the image of $G$ under the $d$-th symmetric power of the transpose to 
$S \hookrightarrow  \cO_{\Gr(2,n+1)}^{n+1}$.

\begin{prop} \cite[Th.~1.10]{AK}
Let $X\subset \bP^n, n\ge 3,$ be a smooth cubic hypersurface. 
Then $F_1(X)$ is smooth of dimension $2n-6$.
\end{prop}
The proof is a local computation on tangent spaces.

\begin{prop} \label{prop:degreetwo}
Let $\ell \subset X \subset \bP^n$ be a smooth cubic hypersurface 
containing a line. Then $X$ admits a degree two unirational
parametrization, i.e., a degree two mapping
$$\rho:\bP^{n-1} \dashrightarrow X.$$
\end{prop}
Since this result is classical we only sketch the key ideas.
Consider the diagram
$$
\xymatrix{
& {\Bl_{\ell}(\bP^n)} \ar[dl] \ar[dr]& \\
{\bP^n}  \ar@{-->}[rr] & &   {\bP^{n-2}}
}
$$
where the bottom arrow is projection from $\ell$. 
The right arrow is a $\bP^2$ bundle.
This induces
$$
\xymatrix{
& {\Bl_{\ell}(X)} \ar[dl] \ar[dr]^q& \\
X  \ar@{-->}[rr] & &   {\bP^{n-2}}
}
$$
where $q$ is a conic bundle. The exceptional divisor
$E\simeq \bP(N_{\ell/X})\subset \Bl_{\ell}(X)$ 
meets each conic fiber in two points.
Thus the base change
$$Y:=\Bl_{\ell}(X)\times_{\bP^{n-2}}E \ra E$$
has a rational section and we obtain birational equvialences
$$Y\stackrel{\sim}{\dashrightarrow} \bP^1 \times E 
\stackrel{\sim}{\dashrightarrow} \bP^1 \times \bP^{n-3} \times \ell
\stackrel{\sim}{\dashrightarrow} \bP^{n-1}.$$
The induced $\rho:Y \dashrightarrow X$ is generically finite of degree two.

\section{Special cubic fourfolds}
 
We use the terminology {\em very general} to mean `outside
the union of a countable collection of Zariski-closed subvarieties'.
Throughout this section, $X$ denotes a smooth cubic fourfold over $\bC$.

\subsection{Structure of cohomology}
\label{subsect:SCD}

Let $X$ be a cubic fourfold and $h\in H^2(X,\bZ)$ the Poincar\'e
dual to the hyperplane class, so that $h^4=\deg(X)=3$. 
The Lefschetz hyperplane theorem and Poincar\'e duality give
$$H^2(X,\bZ)=\bZ h , \quad H^6(X,\bZ)=\bZ \frac{h^3}{3}.$$
The Hodge numbers of $X$ take the form
$$\begin{array}{ccccccccc}
	&	&	&	& 1 & 	&	&	& \\
	& 	&       & 0     &   & 0 &       &       & \\
	&	&  0	& 	& 1 & 	&   0	&	& \\
	&0	&       & 0     &   & 0 &       &   0   & \\
0 &	&  1	& 	& 21 & 	&   1	&	& 0 
\end{array}
$$
so the Hodge-Riemann bilinear relations imply
that $H^4(X,\bZ)$ is a unimodular lattice under the
intersection form $\left<,\right>$ of signature $(21,2)$. 
Basic classification results on quadratic forms \cite[Prop.~2.1.2]{HaCM}
imply
$$L:=H^4(X,\bZ)_{\left<,\right>} \simeq (+1)^{\oplus 21} \oplus (-1)^{\oplus 2}.$$
The primitive cohomology
$$L^0 := \{ h^2 \}^{\perp} \simeq A_2 \oplus U^{\oplus 2} \oplus E_8^{\oplus 2},$$
where
$$A_2=\left(\begin{matrix} 2 & 1 \\ 1 & 2 \end{matrix} \right), 
\quad
U=\left(\begin{matrix} 0 & 1 \\ 1 & 0 \end{matrix} \right), $$
and $E_8$ is the positive definite lattice associated with the 
corresponding Dynkin diagram.  
This can be shown using the incidence correspondence between $X$ and 
its variety of lines $F_1(X)$, which induces the
{\em Abel-Jacobi} mapping \cite{BD}
\begin{equation} \label{eqn:AbelJacobi}
\alpha:H^4(X,\bZ) \ra  H^2(F_1(X),\bZ)(-1),
\end{equation}
compatible with Hodge filtrations. (See \cite[\S2]{AdTh} for another 
approach.)
Restricting to primitive cohomology gives an isomorphism
$$
\alpha:H^4(X,\bZ)_{\prim} \stackrel{\sim}{\lra} H^2(F_1(X),\bZ)_{\prim}(-1).
$$
Note that $H^2(F_1(X),\bZ)$
carries the {\em Beauville-Bogomolov} form $\left(,\right)$;
see \S\ref{subsect:inter} for more discussion.
The shift in weights explains the change in signs
$$\left(\alpha(x_1),\alpha(x_2)\right)=
-\left<x_1,x_2\right>.$$

\subsection{Special cubic fourfolds defined}
\label{subsect:SCFD}

For a very general cubic fourfold $X$, any algebraic surface $S\subset X$
is homologous to a complete intersection, i.e.,
$$H^{2,2}(X,\bZ):=H^4(X,\bZ) \cap H^2(\Omega^2_X) \simeq \bZ h^2$$
so 
$$[S]\equiv n h^2, \quad n=\deg(S)/3.$$
This follows from the Torelli Theorem and the irreducibility
of the monodromy action for cubic fourfolds \cite{Voisin86};
see \S~\ref{subsect:structural} below for more details.
In particular, $X$ does not contain any quartic scrolls
or any surfaces of degree four; this explains why Morin's rationality
argument could not be correct.

On the other hand, the integral Hodge conjecture is valid
for cubic fourfolds \cite[Th.~1.4]{VoisinJAG}, so every class
$$\gamma \in H^{2,2}(X,\bZ)$$
is algebraic, i.e., arises from a codimension two algebraic cycle
with integral coefficients. Thus if
$$H^{2,2}(X,\bZ) \supsetneq \bZ h^2$$
then $X$ admits surfaces that are not homologous to complete
intersections.

\begin{defi}
A cubic fourfold $X$ is {\em special} if it admits an algebraic
surface $S\subset X$ not homologous to a complete intersection.
A {\em labelling} of a special cubic fourfold consists of 
a rank two saturated sublattice
$$h^2 \in K \subset H^{2,2}(X,\bZ);$$
its {\em discriminant} is the determinant of the intersection
form on $K$.
\end{defi}

Let $S\subset X$ be a smooth surface. Recall that 
$$c_1(T_X)=3h, \quad c_2(T_X)=6h^2$$
so the self-intersection
$$
\begin{array}{rcl}
\left<S,S\right>&=&c_2(N_{S/X})=
c_2(T_X|S)-c_2(T_S)-c_1(T_S)c_1(T_X|S)+c_1(T_S)^2 \\
                &=& 6H^2+3HK_S+K_S^2-\chi(S), 
\end{array}$$
where $H=h|S$ and $\chi(S)$ is the topological Euler characteristic.

\begin{enumerate}
\item{When $X$ contains a plane $P$ we have
$$K_8=\begin{array}{r|cc}
	& h^2 & P \\ 
\hline 
h^2 & 3    & 1 \\
P  & 1    & 3 
\end{array}.
$$
}
\item{When $X$ contains a cubic scroll $\Sigma_3$, i.e.,
$\Bl_p(\bP^2)$ embedded in $\bP^4$, we have
$$K_{12}=\begin{array}{r|cc}
	& h^2 & \Sigma_3 \\ 
\hline 
h^2 & 3    & 3 \\
\Sigma_3  & 3    & 7 
\end{array}.
$$
}
\item{When $X$ contains a quartic scroll $\Sigma_4$ or a quintic
del Pezzo surface $T$ then we have
$$K_{14}=\begin{array}{r|cc}
	& h^2 & \Sigma_4 \\ 
\hline 
h^2 & 3    & 4 \\
\Sigma_4  & 4    & 10 
\end{array} \simeq
\begin{array}{r|cc}
	& h^2 & T \\ 
\hline 
h^2 & 3    & 5 \\
T  & 5    & 13
\end{array},
 \quad T=3h^2-\Sigma_4.
$$
}
\end{enumerate}

We return to cubic fourfolds containing two disjoint planes $P_1$ and
$P_2$. Here we have a rank three lattice of algebraic classes,
containing a distinguished rank two lattice:
$$
\begin{array}{r|ccc}
	& h^2 & P_1 & P_2 \\ 
\hline 
h^2 & 3    & 1 & 1  \\
P_1  & 1 & 3 & 0 \\
P_2  & 1 & 0 & 3 
\end{array} \supset
\begin{array}{r|cc}
	& h^2 & \Sigma_4  \\ 
\hline 
h^2 & 3    & 4 \\
\Sigma_4  & 4    & 10
\end{array}, \quad \Sigma_4=2h^2 - P_1 - P_2.
$$

\subsection{Structural results}
\label{subsect:structural}
Voisin's Torelli Theorem and the geometric description of
the period domains yields a qualitative description
of the special cubic fourfolds. 

Consider cubic
fourfolds $X$ with a prescribed saturated sublattice 
$$h^2 \in M\subset L\simeq H^4(X,\bZ)$$
of algebraic classes. The Hodge-Riemann bilinear relations imply
that $M$ is positive definite. Then the Hodge structure on $X$ is 
determined by 
$H^1(X,\Omega^3_X) \subset M^{\perp}\otimes \bC,$
which is isotropic for $\left<,\right>$. The relevant period
domain is 
$$\cD_M=\{ [\lambda] \in \bP(M^{\perp}\otimes \bC): \left<\lambda,\lambda \right> =0 \},$$
or at least the connected component 
with the correct orientation. (See \cite[\S4]{MarkSurv}
for more discussion of orientations.)
The Torelli theorem \cite{Voisin86} asserts that the period map
$$\begin{array}{rcl}
\tau: \cC & \ra & \Gamma \backslash \cD_{\bZ h^2}\\
        X& \mapsto & H^1(X,\Omega^3_X)
\end{array}$$
is an open immersion; $\Gamma$ is the group of automorphisms
of the primitive cohomology lattice $L^0$ arising from the monodromy
of cubic fourfolds.
Cubic fourfolds with additional algebraic cycles,
indexed by a saturated sublattice
$$M\subsetneq M' \subset L,$$
correspond to the linear sections of $\cD_M$ of codimension 
$\operatorname{rank}(M'/M)$.

\begin{prop} \label{prop:rough}
Fix a positive definite lattice $M$ of rank $r$
admitting a saturated embedding
$$h^2 \in M \subset L.$$
If this exists then $M^0=\{h^2\}^{\perp} \subset M$
is necessarily even, as it embeds in $L^0$.

Let $\cC_M\subset \cC$ denote the cubic fourfolds $X$ admiting algebraic
classes with this lattice structure
$$ h^2 \in M \subset H^{2,2}(X,\bZ) \subset L.$$
Then $\cC_M$ has codimension $r-1$, provided it is non-empty.
\end{prop}

We can make this considerably more precise in rank two.
For each labelling $K$, pick a generator $K\cap L^0=\bZ v$.
Classifying orbits of primitive positive vectors $v\in L^0$
under the automorphisms of this lattice associated with the monodromy
representation yields:
\begin{theo}\label{theo:irred} \cite[\S3]{HaCM}
Let $\cC_d \subset \cC$ denote the special cubic fourfolds
admitting a labelling of discriminant $d$. Then $\cC_d$
is non-empty if and only if
$d\ge 8$ and $d\equiv 0,2 \pmod{6}$. Moreover, $\cC_d$ is an
irreducible divisor in $\cC$.
\end{theo}
Fix a discriminant
$d$ and consider the locus $\cC_d\subset \cC$.
The Torelli Theorem implies that
irreducible components of $\cC_d$ correspond to 
saturated rank two sublattices realizations
$$h^2 \in K \subset L$$
up to monodromy. The monodromy of cubic fourfolds acts on $L$
via $\Aut(L,h^2)$, the automorphisms of the lattice $L$
preserving $h^2$. Standard results on
embeddings of lattices imply there is a unique 
$K\subset L$ modulo $\Aut(L,h^2)$. The monodromy
group is an explicit finite index subgroup of $\Aut(L,h^2)$,
which still acts transitively on these sublattices.
Hence $\cC_d$ is irreducible.

The rank two lattices associated with labellings
of discrimiant $d$ are:
$$K_d := \begin{cases}
\begin{array}{r|cc}
		& h^2 & S \\
	\hline 
	h^2 & 3 & 1 \\
	S & 1 & 2n+1 \end{array}  & \text{ if } d=6n+2,n\ge 1\\
			& \\
\begin{array}{r|cc}
		& h^2 & S \\
	\hline 
	h^2 & 3 & 0 \\
	S & 0 & 2n \end{array}  & \text{ if } d=6n, n\ge 2.
\end{cases}
$$
The cases $d=2$ and $6$ 
\begin{equation} \label{eqn:k26}
K_2 = 
\begin{array}{r|cc}
		& h^2 & S \\
	\hline 
	h^2 & 3 & 1 \\
	S & 1 & 1 \end{array} \quad
K_6 = \begin{array}{r|cc}
		& h^2 & S \\
	\hline 
	h^2 & 3 & 0 \\
	S & 0 & 2 
\end{array} 
\end{equation}
correspond to limiting Hodge structures 
arising from singular cubic fourfolds: the symmetric
determinant cubic fourfolds \cite[\S4.4]{HaCM}
and the cubic fourfolds with an
ordinary double point \cite[\S4.2]{HaCM}.
The non-special cohomology lattice $K_d^{\perp}$ is also well-defined
for all $(X,K_d) \in \cC_d$. 

Laza \cite{LazaAnnals}, building on work of Looijenga \cite{Loo},
gives precise necessary
and sufficient conditions for when the
$\cC_M$ in 
Proposition~\ref{prop:rough} are nonempty:
\begin{itemize}
\item{$M$ is positive definite and admits
a saturated embedding
$h^2 \in M \subset L$;}
\item{there exists no sublattice 
$h^2 \in K \subset M$ with $K \simeq K_2$ or $K_6$ as in (\ref{eqn:k26}).}
\end{itemize}
Detailed descriptions of the possible lattices
of algebraic classes are given by Mayanskiy \cite{Maya}.
Furthermore, Laza obtains a characterization of the image of the
period map for cubic fourfolds: it is complement of the divisors 
parametrizing `special' Hodge structures with a labelling of discriminant
$2$ or $6$. 

\begin{rema}
Li and Zhang \cite{LZ} have found a beautiful generating series for the degrees
of special cubic fourfolds of discriminant $d$, expressed via modular forms.
\end{rema}
 
We have seen concrete descriptions of surfaces arising in special cubic
fourfolds for $d=8,12,14$. Nuer \cite[\S3]{Nuer}
writes down explicit smooth rational surfaces
arising in generic special cubic fourfolds of discriminants $d\le 38$.
These are blow-ups of the plane at points in general position,
embedded via explicit linear series, e.g.,
\begin{enumerate}
\item{For $d=18$, let $S$ be a generic projection into $\bP^5$ of a sextic
del Pezzo surface in $\bP^6$.}
\item{For $d=20$, let $S$ be a Veronese embedding of $\bP^2$.}
\end{enumerate}
\begin{ques} \label{prob:smoothrat}
Is the algebraic cohomology of a special cubic fourfold 
generated by the classes
of smooth rational surfaces?
\end{ques}
Voisin has shown that the cohomology can be generated by smooth surfaces
(see the proof of \cite[Th.~5.6]{VoisinJEMS}) {\em or} by possibly singular
rational surfaces \cite{VoisinJJM}. 
Low discriminant examples suggest we might be able to achieve both.

A by-product of Nuer's approach, where it applies, 
is to prove that the $\cC_d$ are unirational.
However, for $d\gg 0$ the loci $\cC_d$ are of
general type \cite{TVA}. So a different approach is needed in general.

\subsection{Census of rational cubic fourfolds}
Using this framework, we enumerate the smooth cubic fourfolds known to be
rational:
\begin{enumerate}
\item{cubic fourfolds in $\cC_{14}$, the closure of the Pfaffian
locus;}
\item{cubic fourfolds in $\cC_8$, the locus containing a plane $P$,
such that there exists a class $W$ such that $\left<W, (h^2-P)\right>$ is
odd.}
\end{enumerate}
For the second case, note that the discriminant of the lattice
$M=\bZ h^2 + \bZ P + \bZ W$ has the same parity as 
$\left<W,(h^2-P)\right>$. 

Thus all the cubic fourfolds proven to be rational are parametrized by
one divisor $\cC_{14}$ and a countably-infinite union of codimension two
subvarieties $\cC_M \subset \cC_8$.

\begin{ques} 
Is there a rational (smooth) cubic fourfold not in the enumeration above?
\end{ques}

There are conjectural frameworks
(see \S\ref{subsect:KuzRema} and also \S\ref{subsect:GS})
predicting that many special cubic fourfolds should
be rational. However, few new examples of cubic
fourfolds have been found to support these frameworks.

\section{Associated K3 surfaces}

\subsection{Motivation}
\label{subsect:motivation}
The motivation for considering associated K3 surfaces comes from the 
Clemens-Griffiths \cite{CG} proof of the irrationality of cubic threefolds.
Suppose $X$ is a rational threefold. Then we have an
isomorphism of polarized Hodge structures
$$H^3(X,\bZ)=\oplus_{i=1}^n H^1(C_i,\bZ)(-1)$$
where the $C_i$ are smooth projective curves.
Essentially, the $C_i$ are blown up in the birational map
$$\bP^3 \dashrightarrow X.$$

If $X$ is a rational {\em fourfold} then we can look for the cohomology
of surfaces blown up in the birational map
$$\rho: \bP^4 \dashrightarrow X.$$
Precisely, if $P$ is a smooth projective fourfold, $S\subset P$
an embedded surface, and $\tilde{P}=\Bl_S(P)$ then we have
\cite[\S6.7]{Fulton}
$$H^4(\tilde{P},\bZ)\simeq H^4(P,\bZ) \oplus_{\perp} H^2(S,\bZ)(-1).$$
The homomorphism $H^4(P,\bZ) \ra H^4(\tilde{P},\bZ)$ is induced by
pull-back; the homomorphism $H^2(S,\bZ)(-1) \ra H^4(\tilde{P},\bZ)$
comes from the composition of pull-back and push-forward
$$\xymatrix{
E=\bP(N_{S/P}) \ar@{^(->}[r] \ar[d] & \tilde{P} \\
S              &		 
}
$$
where $N_{S/P}$ is the normal bundle and $E$ the exceptional divisor.
Blowing up points in $P$ contributes Hodge-Tate summands
$\bZ(-2)$ with negative self-intersection to its middle cohomology;
these have the same affect as blowing up a surface (like $\bP^2$)
with $H^2(S,\bZ)\simeq \bZ$.
Blowing up curves does not affect middle cohomology.

Applying the Weak Factorization Theorem \cite{Wlo,AKMW}---that every
birational map is a composition of blow-ups and blow-downs 
along smooth centers---we
obtain the following:
\begin{prop}
Suppose $X$ is a rational fourfold. 
Then there exist smooth projective surfaces
$S_1,\ldots,S_n$ and $T_1,\ldots,T_m$ such that we have an isomorphism
of Hodge structures
$$H^4(X,\bZ) \oplus (\oplus_{j=1,\ldots,m} H^2(T_j,\bZ)(-1)) \simeq 
\oplus_{i=1,\ldots,n} H^2(S_i,\bZ)(-1).$$
\end{prop}
Unfortunately, it is not clear how to extract a computable
invariant from this observation; but see \cite{ABvB,Kul} for
work in this direction.

\subsection{Definitions and reduction to lattice
equivalences}
In light of examples illustrating how rational cubic fourfolds
tend to be entangled with K3 surfaces, it is natural to explore
this connection on the level of Hodge structures.

Let $(X,K)$ denote a labelled special cubic fourfold.
A polarized K3 surface $(S,f)$ is {\em associated} 
with $(X,K)$ if there exists an isomorphism of lattices
$$H^4(X,\bZ)\supset 
K^{\perp} \stackrel{\sim}{\lra} f^{\perp} \subset H^2(S,\bZ)(-1)$$
respecting Hodge structures.

\begin{exam}[Pfaffians]
Let $X$ be a Pfaffian cubic fourfold (see \S\ref{subsect:Pfaff}) and 
$$K_{14}=\begin{array}{r|cc}
	& h^2 & \Sigma_4 \\ 
\hline 
h^2 & 3    & 4 \\
\Sigma_4  & 4    & 10 
\end{array} \simeq
\begin{array}{r|cc}
	& h^2 & T \\ 
\hline 
h^2 & 3    & 5 \\
T  & 5    & 13
\end{array},
$$
the lattice containing the classes of the resulting
quartic scrolls and quintic del Pezzo
surfaces.
Let $(S,f)$ be the K3 surface of degree $14$
arising from the Pfaffian construction.
Then $(S,f)$ is associated with $(X,K_{14})$.
\end{exam}

\begin{exam}[A suggestive non-example]
As we saw in \S\ref{subsect:Pfaff}, a cubic fourfold $X$ containing a plane $P$
gives rise to a degree two K3 surface $(S,f)$. However, this
is {\em not} generally associated with the cubic fourfold. If $K_8\subset H^4(X,\bZ)$
is the labelling then
$$K_8^{\perp} \subset f^{\perp} $$
as an index two sublattice \cite[\S9.7]{vanGeemen}. However,
when the quadric bundle $q:\Bl_P(X)\ra \bP^2$ admits a section
(so that $X$ is rational),
$S$ often admits a polarization $g$ such that
$(S,g)$ is associated with some labelling of $X$.
See \cite{Kuz08,Kuz15}
for further discussion.
\end{exam}

\begin{prop}
The existence of an associated K3 surface depends only
on the discriminant of the rank two lattice $K$.
\end{prop}
Here is an outline of the proof; we refer to
\cite[\S5]{HaCM} for details. 

Recall the discussion of Theorem~\ref{theo:irred}
in \S\ref{subsect:structural}: For each 
discriminant $d\equiv 0,2\pmod{6}$ with $d>6$, 
there exists a lattice $K_d^{\perp}$
such that each special cubic fourfold of discriminant $d$
$(X,K)$ has $K^{\perp}\simeq K_d^{\perp}$.
Consider the primitive
cohomology lattice
$$\Lambda_d:=f^{\perp} \subset H^2(S,\bZ)(-1)$$
for a polarized K3 surface $(S,f)$ of degree $d$. The moduli
space $\cN_d$ of such surfaces is connected, so $\Lambda_d$ is
well-defined up to isomorphism. 

We claim $(X,K_d)$ admits an associated K3 surface if and only
if there exists an isomorphism of lattices
$$\iota: K_d^{\perp} \stackrel{\sim}{\lra}
\Lambda_d.$$
This is clearly a necessary condition. For sufficiency, 
given a Hodge structure on $\Lambda_d$ surjectivity of the Torelli map
for K3 surfaces \cite{Siu} ensures there exists a K3 surface $S$ and a
divisor $f$ with that Hodge structure. It remains
to show that $f$ can be taken to be a polarization of $S$,
i.e., there are no $(-2)$-curves orthogonal to $f$.
After twisting and applying $\iota$, any such curve yields
an algebraic class $R\in H^{2,2}(X,\bZ)$ with $\left<R,R\right>=2$
and $\left<h^2,R\right>=0$. In other words, we obtain a labelling
$$K_6=\left<h^2,R\right> \subset H^4(X,\bZ).$$
Such labellings are associated with nodal cubic fourfolds,
violating the smoothness of $X$.

Based on this discussion, it only remains to characterize
when the lattice isomorphism exists. Nikulin's theory
of discriminant forms \cite{Nik} yields:
\begin{prop} \cite[Prop.~5.1.4]{HaCM}:
Let $d$ be a positive integer congruent to $0$ or $2$ modulo $6$.
Then there exists an isomorphism
$$\iota: K_d^{\perp} \stackrel{\sim}{\lra}
\Lambda_d(-1)$$
if and only if $d$ is not divisible by $4,9$ or any odd
prime congruent to $2$ modulo $3$.
\end{prop}
\begin{defi}
An even integer $d>0$ is {\em admissible} if it is not divisble
by $4,9$ or
any odd prime congruent to $2$ modulo $3$.
\end{defi}
Thus we obtain:
\begin{theo} \label{theo:K3assoc}
A special cubic fourfold $(X,K_d)$ admits an associated K3
surface if and only if $d$ is admissible.
\end{theo}

\subsection{Connections with Kuznetsov's philosophy}
\label{subsect:KuzRema}
Kuznetsov has proposed a criterion for rationality expressed 
via derived categories \cite[Conj.~1.1]{Kuz08} \cite{Kuz15}:
Let $X$ be a cubic fourfold, $\cD^b(X)$ the
bounded derived category of coherent sheaves on $X$,
and $\cA_X$ the subcategory orthogonal to the exceptional
collection
$\{ \cO_X,\cO_X(1),\cO_X(2) \}$.
Kuznetsov proposes that $X$ is rational if and only if
$\cA_X$ is equivalent to the derived category of a 
K3 surface. He verifies this for the known examples
of rational cubic fourfolds.

Addington and Thomas \cite[Th.~1.1]{AdTh}
show that the generic $(X,K_d)\in \cC_d$ satisfies
Kuznetsov's derived category condition precisely
when $d$ is admissible. In \S\ref{subsect:AT}
we present some of the geometry behind this result. 
Thus we find:
\begin{quote}
Kuznetsov's conjecture would imply that 
the generic $(X,K_d) \in \cC_d$ for admissible
$d$ is rational.
\end{quote}
In particular, special cubic fourfolds of discriminants
$d=26, 38, 42, \ldots$ would all be rational!

\subsection{Naturality of associated K3 surfaces}
There are {\em a priori} many choices for the 
lattice isomorphism $\iota$ so a given cubic fourfold
could admit several associated K3 surfaces. Here we will
analyze this more precisely. Throughout, assume that $d$ 
is admissible.

We require a couple variations on $\cC_d \subset \cC$:
\begin{itemize}
\item{Let $\cC_d^{\nu}$ denote labelled cubic fourfolds,
with a saturated lattice $K\ni h^2$ 
of algebraic classes of rank two and discriminant
$d$.}
\item{Let $\cC_d'$ denote pairs consisting of a
cubic fourfold $X$ and a saturated embedding of $K_d$ into
the algebraic cohomology.}
\end{itemize}
We have natural maps
$$\cC_d' \ra \cC_d^{\nu} \ra \cC_d.$$
The second arrow is normalization
over cubic fourfolds admitting multiple labellings of discriminant
$d$. To analyze the first arrow, note that the $K_d$ admits
non-trivial automorphisms fixing $h^2$ if and only if $6|d$.
Thus the first arrow is necessarily an isomorphism unless
$6|d$. When $6|d$ the lattice $K_d$ admits an automorphism
acting by multiplication by $-1$ on the orthogonal 
complement of $h^2$. $\cC'_d$ is irreducible if this involution can be
realized in the monodromy group. An analysis of the monodromy group
gives:

\begin{prop} \cite[\S5]{HaCM}:
For each admissible $d>6$, 
$\cC_d'$ is irreducible and admits an open immersion
into the moduli space $\cN_d$ of polarized K3 surfaces of degree $d$.
\end{prop}

\begin{coro} \label{coro:twoK3}
Assume $d>6$ is admissible.
If $d\equiv 2\pmod{6}$ then $\cC_d$ is
birational to $\cN_d$. Otherwise $\cC_d$ is
birational to a quotient of $\cN_d$ by an involution.
\end{coro}
Thus for $d=42,78,\ldots$ cubic fourfolds $X\in \cC_d$
admit {\em two} associated K3 surfaces.

Even the open immersions from the double covers
$$j_{\iota,d}:\cC'_d \hookrightarrow \cN_d$$
are typically not
canonical. The possible choices correspond to 
orbits of the isomorphism 
$$\iota: K_d^{\perp} \stackrel{\sim}{\lra}
\Lambda_d(-1)$$
under postcomposition by automorphisms of $\Lambda_d$
coming from the monodromy of K3 surfaces and 
precomposition by automorphisms of $K_d^{\perp}$ coming
from the the subgroup of the monodromy group of cubic
fourfolds fixing the elements of $K_d$. 

\begin{prop} \cite[\S5.2]{HaCM} \label{prop:counting}
Choices of 
$$j_{\iota,d}:\cC'_d \hookrightarrow \cN_d$$
are in bijection with the set
$$\{ a \in \bZ/d\bZ: a^2\equiv 1\pmod{2d} \}/\pm 1.$$
If $d$ is divisible by $r$ distinct odd primes then there
are $2^{r-1}$ possibilities. 
\end{prop}

\begin{rema} \label{rema:DE}
The ambiguity in associated K3 surfaces can be expressed in the
language of equivalence of derived categories. Suppose that $(S_1,f_1)$ and
$(S_2,f_2)$ are polarized K3 surfaces of degree $d$, both associated
with a special cubic fourfold of discriminant $d$.
This means we have an isomorphism of Hodge structures
$$H^2(S_1,\bZ)_{\prim} \simeq H^2(S_2,\bZ)_{\prim}$$
so their transcendental cohomologies are isomorphic.
Orlov's Theorem \cite[\S3]{Orlov} implies $S_1$ and $S_2$ are
{\em derived equivalent}, i.e., their bounded
derived categories of 
coherent sheaves are equivalent.

Proposition~\ref{prop:counting} may be compared with the formula counting
derived equivalent K3 surfaces in \cite{HLOY}.
We will revisit this issue in \S\ref{subsect:AT}.
\end{rema}

\subsection{Interpreting associated K3 surfaces I}
\label{subsect:inter}

We offer geometric interpretations of associated K3 surfaces.
These are most naturally expressed in terms of moduli spaces 
of sheaves on K3 surfaces. 

Let $M$ be an irreducible holomorphic symplectic variety,
i.e., a smooth simply connected projective variety such that
$\Gamma(M,\Omega^2_M)=\bC \omega$ where $\omega$ is everywhere
nondegenerate. The cohomology $H^2(M,\bZ)$ admits a 
distinguished integral quadratic form $\left(,\right)$,
called the {\em Beauville-Bogomolov} form \cite{Beau}.
Examples include:
\begin{itemize}
\item{K3 surfaces $S$ with $\left(,\right)$ the intersection form;}
\item{Hilbert schemes $S^{[n]}$ of length $n$ zero dimensional
subschemes on a K3 surface $S$, with
\begin{equation} \label{eqn:delta}
H^2(S^{[n]},\bZ)\simeq H^2(S,\bZ) \oplus_{\perp} \bZ \delta, \quad
\left(\delta,\delta\right)=-2(n-1),
\end{equation}
where $2\delta$ parametrizes the non-reduced subschemes.}
\end{itemize}

\begin{exam}
Let $(S,f)$ be a generic degree $14$ K3 surface.
Then $S^{[2]}\simeq F_1(X)\subset \Gr(2,6)$ where $X$ is a Pfaffian
cubic fourfold (see \S\ref{subsect:Pfaff}).
The polarization induced from the Grassmannian is
$2f-5\delta$; note that
$$\left(2f-5\delta,2f-5\delta\right)=4\cdot 14 - 25\cdot 2 = 6.$$
\end{exam}

The example implies that
if $F_1(X)\subset \Gr(2,6)$ is the variety of lines on an
arbitrary cubic fourfold then the polarization $g=\alpha(h^2)$ satisfies
$$\left(g,g\right)=6, \quad \left(g,H^2(F_1(X),\bZ)\right)=2\bZ.$$
It follows that the
Abel-Jacobi map 
is an isomorphism of {\em abelian groups}
\begin{equation} 
\label{eqn:AJfull}
\alpha:H^4(X,\bZ) \ra  H^2(F_1(X),\bZ)(-1)
\end{equation}
Indeed, $\alpha$ is an isomorphism on
primitive cohomology and both
$$\bZ h^2 \oplus  H^4(X,\bZ)_{prim} \subset H^4(X,\bZ)$$
and 
$$\bZ g \oplus  H^2(F_1(X),\bZ)_{prim} \subset H^2(F(X),\bZ)$$
have index three as subgroups. 

The Pfaffian case is the first of an infinite series of
examples:
\begin{theo}  \cite[\S6]{HaCM} \cite[Th.~2]{AddMRL}  \label{theo:Fano}
Fix an integer of the form $d=2(n^2+n+1)/a^2$, where
$n>1$ and $a>0$ are integers. Let $X$ be a 
cubic fourfold in $\cC_d$ with variety of lines $F_1(X)$.
Then there exists a polarized K3 surface $(S,f)$ of degree
$d$ and a birational map
$$F_1(X) \stackrel{\sim}{\dashrightarrow}S^{[2]}.$$
If $a=1$ and $X\in \cC_d$ in generic then $F_1(X)\simeq S^{[2]}$
with polarization $g=2f-(2n+1)\delta$.
\end{theo}
The first part relies on Verbitsky's global Torelli theorem for
hyperk\"ahler manifolds \cite{MarkSurv}. The last assertion
is proven via a degeneration/specialization argument along
the nodal cubic fourfolds, which correspond to degree six K3 surfaces
$(S',f')$. We specialize so that $F={S'}^{[2]}$ admits involutions
not arising from involutions of $S'$.
Thus the deformation space of $F$ admits several
divisors parametrizing Hilbert schemes of K3 surfaces.

Since the primitive cohomology of $S$ sits in $H^2(S^{[2]},\bZ)$,
the Abel-Jacobi map (\ref{eqn:AbelJacobi}) explains
why $S$ is associated with $X$. 
If $3|d$ ($d\neq 6$) then Theorem~\ref{theo:Fano}
and Corollary~\ref{coro:twoK3} yield two
K3 surfaces $S_1$ and $S_2$ such that
$$F_1(X) \simeq S_1^{[2]} \simeq S_2^{[2]}.$$

With a view toward extending this argument, we compute the cohomology
of the varieties of lines of special cubic fourfolds.
This follows immediately from (\ref{eqn:AJfull}):
\begin{prop}
Let $(X,K_d)$ be a special cubic fourfold of discriminant $d$,
$F_1(X)\subset \Gr(2,6)$ its variety of lines,
and $g=\alpha(h^2)$ the resulting polarization. 
Then $\alpha(K_d)$ is saturated in $H^2(F_1(X),\bZ)$
and
$$\alpha(K_d)\simeq \begin{cases}
\begin{array}{r|cc}
	& g & T \\
\hline
g      & 6  & 0 \\
T      & 0   & -2n
\end{array} & \text{ if } d=6n \\
\begin{array}{r|cc}
	& g & T \\
\hline
g      & 6  & 2 \\
T      & 2   & -2n
\end{array} & \text{ if } d=6n+2.
\end{cases}
$$
\end{prop}
The following example shows that Hilbert schemes are insufficient
to explain all associated K3 surfaces:
\begin{exam}
Let $(X,K_{74})\in \cC_{74}$ be a generic point, which
admits an associated K3 surface by Theorem~\ref{theo:K3assoc}.
There does not exist a K3 surface $S$ with
$F_1(X) \simeq S^{[2]}$, even birationally.
Indeed, the $H^2(M,\bZ)$ is a birational invariant of holomorphic
symplectic manifolds but 
$$\alpha(K_{74}) \simeq \left( \begin{matrix} 6 & 2 \\
						2 & -24 \end{matrix} \right)
$$
is not isomorphic to the Picard lattice of the Hilbert scheme
$$\Pic(S^{[2]}) \simeq \left( \begin{matrix} 74 & 0 \\
						0 & -2 \end{matrix} \right).$$
Addington \cite{AddMRL} gives a systematic discussion
of this issue.
\end{exam}

\subsection{Derived equivalence and Cremona transformations}
\label{subsect:GS}
In light of the ambiguity of associated K3 surfaces
(see Remark~\ref{rema:DE}) and 
the general discussion in \S\ref{subsect:motivation},
it is natural to seek diagrams
\begin{equation} \label{eqn:diagram}
\xymatrix{
  	&  \Bl_{S_1}(\bP^4)\ar[dl]_{\beta_1} \ar[r]^{\sim} & \Bl_{S_2}(\bP^4) \ar[dr]^{\beta_2}   & \\
  \bP^4 &  &   & \bP^4
}
\end{equation}
where $\beta_i$ is the blow-up along a smooth surface $S_i$, with
$S_1$ and $S_2$ derived equivalent but not birational.

Cremona transformations of $\bP^4$ with smooth surfaces as their centers
have been classified by Crauder and
Katz \cite[\S 3]{CrKa}; possible centers are either quintic elliptic scrolls or
surfaces $S\subset \bP^4$ of degree ten given by the vanishing of
the $4\times 4$ minors
of a $4\times 5$ matrix of linear forms. 
A generic surface of the
latter type admits divisors
$$ \begin{array}{c|cc}
	& K_S & H \\
\hline
K_S & 5 & 10 \\
H   & 10 & 10
\end{array}, 
$$
where $H$ is the restriction of the hyperplane class from $\bP^4$;
this lattice admits an involution fixing $K_S$ with $H \mapsto 4K_S-H$.
See \cite[Prop.~9.18ff.]{RanThesis}, \cite{RanPaper}, and \cite[p.280]{Baker} for discussion of these surfaces. 

The Crauder-Katz classification therefore precludes
diagrams of the form (\ref{eqn:diagram}).
We therefore recast our search as follows:
\begin{ques} \label{ques:Cremona}
Does there exist a diagram
$$\xymatrix{
  	&  X\ar[dl]_{\beta_1}  \ar[dr]^{\beta_2}   & \\
  \bP^4 &     & \bP^4
}
$$
where $X$ is smooth and the $\beta_i$ are birational projective morphisms,
and K3 surfaces $S_1$ and $S_2$ such that
\begin{itemize}
\item{$S_1$ and $S_2$ are derived equivalent but not isomorphic;}
\item{$S_1$ is birational to a center of $\beta_1$ but not to any center
of $\beta_2$;}
\item{$S_2$ is birational to a center of $\beta_2$ but not to any center
of $\beta_1$?}
\end{itemize}
\end{ques}

This could yield counterexamples to the Larsen-Lunts cut-and-paste
question on Grothendieck groups, similar to those found by Borisov \cite{Bor}. 
Galkin and Shinder 
\cite[\S7]{GalShi} showed that {\em if} the class of
the affine line were a non-zero divisor in the Grothendieck
group then for each rational cubic fourfold $X$ the variety of lines
$F_1(X)$ would be birational to $S^{[2]}$ for some K3 surface
$S$. (Note that Borisov shows it {\em is} a zero
divisor.)  We have seen (Theorem~\ref{theo:Fano}) that this
condition holds for infinitely many $d$.

\subsection{Interpreting associated K3 surfaces II}
\label{subsect:AT}

Putting Theorem~\ref{theo:Fano}
on a general footing requires a larger inventory of varieties
arising from K3 surfaces. We shall use
fundamental results on moduli spaces of sheaves on K3 surfaces
due to Mukai \cite{MukaiTata}, Yoshioka, and others.
Let $S$ be a complex projective K3 surface. The
{\em Mukai lattice}
$$\tilde{H}^*(S,\bZ)= H^0(S,\bZ)(-1) \oplus 
H^2(S,\bZ) \oplus H^4(S,\bZ)(1)$$
with unimodular form
$$\left((r_1,D_1,s_1),(r_2,D_2,s_2)\right)=
-r_1s_2+D_1D_2-r_2s_1$$
carries the structure of a Hodge structure of weight
two. (The zeroth and fourth cohomology are of type $(1,1)$
and the middle cohomology carries its standard Hodge
structure.) 
Thus we have
$$\tilde{H}^*(S,\bZ) \simeq H^2(S,\bZ) \oplus_{\perp} 
\left( \begin{matrix} 0 & -1 \\
			-1 & 0 \end{matrix} \right).$$

Suppose $v=(r,D,s) \in \tilde{H}^*(S,\bZ)$ is primitive of type 
$(1,1)$ with $\left(v,v\right)\ge 0$.
Assume that one of the following holds:
\begin{itemize}
\item{$r>0$;}
\item{$r=0$ and $D$ is ample.}
\end{itemize}
Fixing a polarization $h$ on $S$, we may consider
the moduli space $M_v(S)$ of sheaves Gieseker stable
with respect to $h$. Here $r$ is the rank, $D$ is the first Chern
class, and $r+s$ is the Euler characteristic.
For $h$ chosen suitably general
(see \cite[\S0]{YoMA} for more discussion), $M_v(S)$ is 
a projective holomorphic symplectic manifold deformation
equivalent to the Hilbert scheme of length $\frac{\left(v,v\right)}{2}+1$
zero dimensional subschemes of a K3 surface \cite[\S8]{YoMA}, \cite[Th.~0.1]{Yoshioka99}. Thus $H^2(M_v(S),\bZ)$ comes with a Beauville-Bogomolov form
and we have an isomorphism of Hodge structures
\begin{equation} \label{eqn:cohomology}
H^2(M_v(S),\bZ)=\begin{cases} v^{\perp}/\bZ v & \text{ if } \left(v,v\right)=0\\
				v^{\perp}  &  \text{ if }\left(v,v\right)>0 
		\end{cases}
\end{equation}
\begin{exam}
The case of ideal sheaves of length two subschemes is $r=1$, $D=0$,
and $s=-1$. Here we recover formula (\ref{eqn:delta})
for $H^2(S^{[2]},\bZ)=H^2(M_{(1,0,-1)}(S),\bZ)$.
\end{exam}

We shall also need recent results of Bayer-Macr\`i \cite{BM1,BM2}:
Suppose that $M$ is holomorphic symplectic and birational
to $M_v(S)$ for some K3 surface $S$. Then we may interpret
$M$ as a moduli space of objects on the derived category of 
$S$ with respect to a suitable Bridgeland stability condition
\cite[Th.~1.2]{BM2}.

Finally, recall Nikulin's approach to 
lattice classification and embeddings \cite{Nik}. 
Given an even unimodular lattice $\Lambda$ and a primitive
nondegenerate sublattice $N\subset \Lambda$, the discriminant
group $d(N):=N^*/N$ is equipped with a $(\bQ/2\bZ)$-valued
quadratic form $q_N$, which encodes most of the $p$-adic
invariants of $N$. The orthogonal complement $N^{\perp}\subset \Lambda$ 
has related invariants
$$(d(N^{\perp}),-q_{N^{\perp}}) \simeq
(d(N),q_N).$$
Conversely, given a pair of nondegenerate even
lattices with complementary invariants,
there exists a unimodular even lattice containing them as orthogonal
complements \cite[\S12]{Nik}.

\begin{theo} \label{theo:AT}
Let $(X,K_d)$ denote a labelled special cubic fourfold
of discriminant $d$. Then $d$ is admissible if and only if
there exists a polarized K3 surface
$(S,f)$, a Mukai vector 
$v=(r,af,s) \in \tilde{H}(S,\bZ)$, a stability condition $\sigma$,
and an isomorphism
$$\varpi: M_v(S) \stackrel{\sim}{\lra} F_1(X)$$
from the moduli space of objects in the derived category
stable with respect to $\sigma$, inducing
an isomorphism between the primitive cohomology of $(S,f)$
and the twist of the non-special cohomology of $(X,K_d)$.
\end{theo}
This is essentially due to Addington and Thomas \cite{AdTh,AddMRL}.

\begin{proof}
Let's first do the reverse direction; this gives us an opportunity
to unpack the isomorphisms in the statement. 
Assume we have the moduli space and isomorphism as described.
After perhaps applying a shift and taking duals, 
we may assume $r\ge 0$ and $a \ge 0$; if $a=0$ then $v=(1,0,-1)$,
i.e., the Hilbert scheme up to birational equivalence.
We still have the computation (\ref{eqn:cohomology}) of the
cohomology 
$$H^2(M_v(S),\bZ))=v^{\perp} \subset \tilde{H}^*(S,\bZ);$$
see \cite[Th.~6.10]{BM1} for discussion relating 
moduli spaces of Bridgeland-stable objects and Gieseker-stable
sheaves.
Thus we obtain a saturated embedding of the primitive cohomology
of $(S,f)$
$$\Lambda_d \hookrightarrow H^2(M_v(S),\bZ).$$
The isomorphism $\varpi$ allows us to identify this with a sublattice
of $H^2(F(X),\bZ)$ coinciding with $\alpha(K_d)^{\perp}$.
Basic properties of the Abel-Jacobi map (\ref{eqn:AbelJacobi}) 
imply that $(S,f)$ is associated with $(X,K_d)$,
thus $d$ is admissible by Theorem~\ref{theo:K3assoc}.

Now assume $d$ is admissible and consider the lattice $-K_d^{\perp}$,
the orthogonal complement of $K_d$ in the middle
cohomology of a cubic fourfold with the intersection form reversed.
This is an even lattice of signature $(2,19)$.

If $X$ is a cubic fourfold then there is a natural 
primitive embedding of lattices \cite[\S9]{MarkSurv}
$$H^2(F_1(X),\bZ) \hookrightarrow \Lambda$$
where $\Lambda$ is isomorphic to the Mukai lattice of a K3 surface
$$\Lambda = U^{\oplus 4} \oplus (-E_8)^{\oplus 2}.$$
Here `natural' means that the monodromy representation on 
$H^2(F_1(X),\bZ)$ extends naturally to $\Lambda$.

Now consider the orthogonal complement $M_d$ to
$-K_d^{\perp}$ {\em in the Mukai lattice} $\Lambda$. 
Since $d$ is admissible 
$$-K_d^{\perp} \simeq \Lambda_d \simeq (-d) \oplus U^{\oplus 2} \oplus
(-E_8)^{\oplus 2}$$
so $d(-K_d^{\perp})\simeq \bZ/d\bZ$ and $q_{-K_d^{\perp}}$ takes 
value $-\frac{1}{d}\pmod{2\bZ}$ on one of the generators. 
Thus $d(M_d)=\bZ/d\bZ$ and takes value $\frac{1}{d}$ on one
of the generators. There is a distinguished lattice with these
invariants
$$(d) \oplus U.$$
Kneser's Theorem \cite[\S13]{Nik} implies there is a unique
such lattice, i.e., $M_d\simeq (d) \oplus U$.

Thus for each generator $\gamma \in d(-K_d^{\perp})$ with
$\left(\gamma,\gamma\right)=-\frac{1}{d}\pmod{2\bZ}$,
we obtain an isomorphism of Hodge structures
$$H^2(F_1(X),\bZ) \subset \Lambda\simeq \tilde{H}^*(S,\bZ)$$
where $(S,f)$ is a polarized K3 surface of degree $d$.
Here we take $f$ to be one of generators of $U^{\perp} \subset M_d$.
Let $v\in \Lambda$ generate the orthogonal complement to $H^2(F_1(X),\bZ)$;
it follows that $v=(r,af,s) \in \tilde{H}^*(S,\bZ)$ and after reversing
signs we may
assume $r\ge 0$. 

Consider a moduli space $M_v(S)$ of sheaves stable with respect
to suitably generic polarizations on $S$. Our lattice analysis
yields 
$$\phi:H^2(F_1(X),\bZ)\stackrel{\sim}{\lra} H^2(M_v(S),\bZ),$$
an isomorphism of Hodge structures taking $\alpha(K_d^{\perp})$
to the primitive cohomology of $S$.
The Torelli Theorem \cite[Cor.~9.9]{MarkSurv} yields a birational map
$$\varpi_1:M_v(S) \stackrel{\sim}{\dashrightarrow} F_1(X);$$
since both varieties are holomorphic symplectic, there
is a natural induced isomorphism \cite[Lem.~2.6]{HuyInv}
$$\varpi_1^*:H^2(F_1(X),\bZ) \stackrel{\sim}{\lra} H^2(M_v(S),\bZ),$$
compatible with Beauville-Bogomolov forms and Hodge structures.

{\em A priori} $\phi$ and $\varpi_1^*$ might differ by
an automorphism of the cohomology of $M_v(S)$.
If this automorphism permutes the two connected components of
the positive cone in $H^2(M_v(S),\bR)$, we may reverse the sign
of $\phi$. If it fails to preserve the moving cone, we can apply a
sequence of monodromy reflections on $M_v(S)$ until this is the case
\cite[Thm.~1.5,1.6]{MarkSurv}. These are analogues to reflections
by $(-2)$-classes on the cohomology of K3 surfaces and
are explicitly known for
manifolds deformation equivalent to Hilbert schemes on K3 surface;
see \cite{HTGAFA,HTGAFA2} for the case of dimension four and
\cite[\S9.2.1]{MarkSurv} for the general picture.
In this situation, the reflections correspond 
to spherical objects
in the derived category of $S$ orthogonal to $v$, thus give
rise to autoequivalences on the derived category \cite{ST,HLOY2}.
We use these to modify the stability condition on $M_v(S)$.
The resulting
$$\varpi_2^*:H^2(F_1(X),\bZ) \stackrel{\sim}{\lra} H^2(M_v(S),\bZ)$$
differs from $\phi$ by an automorphism
that preserves moving cones, but may not preserve polarizations.
Using \cite[Th.~1.2(b)]{BM2}, we may modify the stability condition on
$M_v(S)$ yet again so that the polarization $g$ on $F_1(X)$ 
is identified with a polarization on $M_v(S)$. Then
the resulting
$$\varpi=\varpi_3:M_v(S) \stackrel{\sim}{\dashrightarrow} F_1(X);$$
preserves ample cones and thus is an isomorphism. Hence $F_1(X)$
is isomorphic to some moduli space of $\sigma$-stable objects over $S$.
\end{proof}

\begin{rema}
As suggested by Addington and Thomas \cite[\S7.4]{AdTh},
it should be possible to employ stability conditions
to show that $\cA_X$ is equivalent to the derived
category of a K3 surface if and only if $X$ admits
an associated K3 surface. (We use the notation
of \S\ref{subsect:KuzRema}.)
\end{rema}

\begin{rema}
Two K3 surfaces $S_1$ 
and $S_2$ are derived equivalent if and only if 
$$\tilde{H}(S_1,\bZ)\simeq \tilde{H}(S_2,\bZ)$$
as weight two Hodge structures \cite[\S3]{Orlov}.
The proof of Theorem~\ref{theo:AT}
explains why the K3 surfaces associated with
a given cubic fourfold are all 
derived equivalent, as mentioned
in Remark~\ref{rema:DE}.
\end{rema}

There are other geometric explanations for K3 surfaces
associated with special cubic fourfolds.
Fix $X$ to be a cubic fourfold not containing a plane. Let 
$M_3(X)$ denote the moduli space of generalized twisted cubics on
$X$, i.e., closed subschemes arising as flat limits of twisted cubics
in projective space.
Then $M_3(X)$ is smooth and irreducible of dimension ten \cite[Thm.~A]{LLSvS}.
Choose $[C] \in M_3(X)$ such that $C$ is a smooth twisted
cubic curve and $W:=\operatorname{span}(C) \cap X$
is a smooth cubic surface. Then the linear series $|\cO_W(C)|$ is 
two-dimensional, so we have a distinguished subvariety
$$[C] \in \bP^2 \subset M_3(X).$$
Then there exists an eight-dimensional irreducible holomorphic symplectic manifold $Z$
and morphisms
$$M_3(X) \stackrel{a}{\lra} Z' \stackrel{\sigma}{\lra} Z$$
where $a$ is an \'etale-locally trivial $\bP^2$ bundle and $\sigma$ is birational 
\cite[Thm.~B]{LLSvS}.
Moreover, $Z$ is deformation equivalent to the Hilbert scheme of length
four subschemes on a K3 surface \cite{AL}.
Indeed, if $X$ is Pfaffian with associated K3 surface $S$ then
$Z$ is birational to $S^{[4]}$. 
It would be useful to have a version of Theorem~\ref{theo:AT}
with $Z$ playing the role of $F_1(X)$.

\section{Unirational parametrizations}
\label{sect:unirat}
We saw in \S\ref{subsect:GGP} that
smooth cubic fourfolds always admit unirational parametrizations
of degree two.
How common are unirational parametrizations of odd degree? 

We review the double point formula \cite[\S9.3]{Fulton}:
Let $S'$ be a nonsingular projective surface, $P$ a
nonsingular projective fourfold, and $f:S' \ra P$ a morphism 
with image $S=f(S')$. We assume that $f:S' \ra S$ is an
isomorphism away from a finite subset of $S'$; equivalently,
$S$ has finitely many singularities with normalization $f:S' \ra S$.
The {\em double point class} $\bD(f) \in \CH_0(S')$ is given
by the formula
$$\begin{array}{rcl}
\bD(f)&=&f^*f_*[S']-(c(f^*T_P)c(T_{S'})^{-1})_2\cap [S'] \\
      &=&f^*f_*[S']-(c_2(f^*T_P)-c_1(T_{S'})c_1(f^*T_P)+c_1(T_{S'})^2
			-c_2(T_{S'})) 
\end{array}
$$
We define
$$D_{S\subset P}=\frac{1}{2}([S]^2_P -
		c_2(f^*T_P)+c_1(T_{S'})c_1(f^*T_P)-c_1(T_{S'})^2
                        +c_2(T_{S'}));$$
if $S$ has just transverse double points then $D_{S\subset P}$
is the number of these singularities.

\begin{exam} (cf.~Lemma~\ref{lemm:scroll})
Let $S'\simeq \bP^1 \times \bP^1\subset \bP^5$ be a quartic scroll,
$P=\bP^4$, and $f:S' \ra \bP^4$ a generic projection. Then
we have 
$$D_{S\subset \bP^4}=\frac{1}{2}(16-40+30-8+4)=1$$
double point. 
\end{exam}

\begin{prop} \label{prop:unirat}
Let $X$ be a cubic fourfold and $S\subset X$ a rational surface of
degree $d$.
Suppose that $S$ has isolated singularities and smooth normalization $S'$,
with invariants 
$D=\deg(S)$, section genus $g$, and self-intersection $\left<S,S\right>$.
If
\begin{equation} \label{eqn:varrho}
\varrho=\varrho(S,X):=\frac{D(D-2)}{2}+(2-2g)-\frac{1}{2}\left<S,S\right> > 0
\end{equation}
then $X$ admits a unirational parametrization $\rho:\bP^4 \dashrightarrow X$ of
degree $\varrho$.
\end{prop}
This draws on ideas from \cite[\S7]{HTGAFA} and \cite[\S5]{VoisinJEMS}.
\begin{proof}
We analyze points $x\in X$ such that the projection
$$f=\pi_x: \bP^5 \dashrightarrow \bP^4$$
maps $S$ birationally to a surface $\hat{S}$
with finitely many singularities, and $f:S\ra \hat{S}$ is finite
and unramified.
Thus $S$ and $\hat{S}$ have the same normalization.

Consider the following conditions:
\begin{enumerate}
\item{$x$ is contained in the tangent space to some singular
point $s\in S$;}
\item{$x$ is contained in the tangent space to some smooth
point $s\in S$;}
\item{$x$ is contained in a
positive-dimensional family of secants to $S$.}
\end{enumerate}
The first condition can be avoided by taking $x$ outside a finite collection
of linear subspaces. 
The second condition can be avoided by taking $x$ outside the tangent variety
of $S$. This cannot concide with a smooth cubic fourfold, which
contains at most finitely many two-planes \cite[Appendix]{BrHB}.
We turn to the third condition. If the secants to $S$ sweep
out a subvariety $Y\subsetneq \bP^5$ then $Y$ cannot be a smooth
cubic fourfold. (The closure of $Y$ contains all the tangent planes to
$S$.) When the secants to $S$ dominate $\bP^5$ then the locus of points 
in $\bP^5$ on 
infinitely many secants has codimension at least two. 

Projection from a point on $X$ outside these three loci
induces a morphism
$$S \ra \hat{S}=f(S)$$
birational and unramified onto its image. Moreover, this image has finitely many
singularities that are resolved by normalization.

Let $W$ denote the second symmetric power of $S$.
Since $S$ is rational, $W$ is rational as well.
There is a rational map coming from residual intersection
$$\begin{array}{rcl}
\rho:W &\dashrightarrow & X \\
s_1+s_2 & \mapsto & x
\end{array}
$$
where $\ell(s_1,s_2)\cap X = \{s_1,s_2,x\}.$
This is well-defined at the generic point of $W$ as the generic
secant to $S$ is not contained in $X$.
(An illustrative special case can be found in Remark~\ref{rema:scroll}.)

The degree of $\rho$ is equal to the number of secants
to $S$ through a generic point of $X$. The analysis
above shows this equals the number of secants to $S$ through
a generic point of $\bP^5$. These in turn correspond
to the number of double points of $\hat{S}$ arising from {\em generic}
projection
to $\bP^4$, i.e.,
$$\deg(\rho)=D_{\hat{S}\subset \bP^4}-D_{S\subset X}.$$

The double point formula gives
$$\begin{array}{rcl}
2D_{S,X}&=&\left<S,S\right>
-(c_2(T_X|S')-c_1(T_{S'})c_1(T_X|S')+c_1(T_{S'})
			-c_2(T_{S'})) \\
2D_{\hat{S},\bP^4}&=& \left<\hat{S},\hat{S}\right>_{\bP^4}
-(c_2(T_{\bP^4}|S')-c_1(T_{S'})c_1(T_{\bP^4}|S')+c_1(T_{S'})
			-c_2(T_{S'})) 
\end{array}
$$
where $\left<\hat{S},\hat{S}\right>_{\bP^4}=D^2$ by Bezout's Theorem.
Taking differences (cf.\S\ref{subsect:SCFD})
yields
$$D_{\hat{S}\subset \bP^4}-D_{S\subset X}=
\frac{1}{2}(D^2-4D+2Hc_1(T_{S'})+\left<S,S\right>),$$
where $H=h|S$.
Using the adjunction formula
$$2g-2=H^2+K_{S'}H$$
we obtain (\ref{eqn:varrho}).
\end{proof}

\begin{coro}[Odd degree unirational parametrizations]
Retain the notation of Proposition~\ref{prop:unirat} and assume that
$S$ is not homologous to a complete intersection. Consider 
the discriminant 
$$d=3\left<S,S\right>-D^2$$
of 
$$
\begin{array}{r|cc}
		& h^2 & S \\
	\hline 
	h^2 & 3 & D \\
	S & D & \left<S,S\right>
\end{array}.
$$
Then the degree
$$
\varrho(S,X)=\frac{d}{2}- 2\left<S,S\right>+(2-2g)+(D^2-D)
$$
has the same parity as $\frac{d}{2}$. Thus the degree of
$\rho:\bP^4 \dashrightarrow X$ is odd
provided $d$ is not divisble by four.
\end{coro}
Compare this with Theorem~\ref{theo:IDOD} below.

How do we obtain surfaces satisfying the assumptions of 
Proposition~\ref{prop:unirat}? Nuer \cite[\S3]{Nuer}
exhibits {\em smooth} such surfaces
for all $d\le 38$, thus we obtain
\begin{coro}
A generic cubic fourfold $X\in \cC_d$, for $d=14, 18, 26, 30,$ and $38$,
admits a unirational parametrization of odd degree.
\end{coro}
There are hueristic 
constructions of such surfaces in far more examples
\cite[\S7]{HTGAFA}. Let $X \in \cC_d$ and consider
its variety of lines $F_1(X)$; for simplicity, assume the Picard group
of $F_1(X)$ has rank two. Recent work of \cite{BMT,BM2} completely
characterizes
rational curves 
$$R\simeq \bP^1 \subset F_1(X)$$
associated with extremal birational contractions of $F_1(X)$.
The incidence correspondence
$$\xymatrix{
\mathcal{INC}\ar[r] \ar[d]  & F_1(X) \\
      X      &      
}
$$
yields 
$$S':=\mathcal{INC}|R \ra S \subset X,$$
i.e., a ruled surface with smooth normalization.
\begin{ques}
When does the resulting ruled surfaces have 
isolated singularities? Is this the case when $R$ is a generic
rational curve arising as an extremal ray of a birational contraction?
\end{ques}
The discussion of \cite[\S7]{HTGAFA} fails to address the
singularity issues, and should be seen as a heuristic approach
rather than a rigorous construction.
The technical issues are illustrated by the following:
\begin{exam}[Voisin's counterexample]
Assume that $X$ is not special so that $\Pic(F_1(X))=\bZ g$.
Let $R$ denote the positive degree generator of the
Hodge classes $N_1(X,\bZ)\subset H_2(F_1(X),\bZ)$; the lattice computations
in \S\ref{subsect:inter} imply that
$$g\cdot R=\frac{1}{2} \left(g,g\right)=3.$$
Moreover, there is a two-parameter family of rational curves
$\bP^1 \subset X, [\bP^1]=R$, corresponding to the cubic surfaces
$S'\subset X$ singular along a line.

These may be seen as follows: The cubic surfaces singular along some
line have codimension seven in the parameter space of all cubic surfaces.
However, there is a nine-parameter family of cubic surface sections
of a given cubic fourfold, parametrized by $\Gr(4,6)$.
Indeed, for a fixed flag
$$\ell \subset \bP^3 \subset \bP^5$$
a tangent space computation shows that the cubic fourfolds 
$$\{X: X\cap \bP^3 \text{ singular along the line } \ell \}$$
dominate the moduli space $\cC$. 

Let $S' \subset \bP^3$ be a cubic surface singular
along a line and $X\supset S'$
a smooth cubic fourfold. Since the generic point of $X$ does not lie on a 
secant line of $S'$, it cannot be used to produce a
unirational parametrization of $X$.
The reasoning for 
Proposition~\ref{prop:unirat} and formula (\ref{eqn:varrho})
is not valid, as $S'$ has
non-isolated singularities.
\end{exam}

Nevertheless, the machinery developed here indicates where
to look for unirational parameterizations of odd degree:
\begin{exam}[$d=42$ case]
Let $X\in \cC_{42}$ be generic.
By Theorem~\ref{theo:Fano}, $F_1(X)\simeq T^{[2]}$ where
$(T,f)$ is K3 surface of degree $42$ and $g=2f-9\delta$.
Take $R$ to be one of the rulings of the divisor in $T^{[2]}$
parametrizing non-reduced subschemes, i.e., those subschemes
supported at a prescribed point of $T$. Note that $R\cdot g = 9$
so the ruled surface $S$ associated with the incidence 
correspondence has numerical invariants:
$$\begin{array}{r|cc}
	& h^2 & S \\
\hline
h^2 & 3 & 9 \\
S   & 9 & 41
\end{array}$$
{\em Assuming} $S$ has isolated singularities, we have
$D_{S\subset X}=8$
and $X$ admits a unirational parametrization of degree
$\varrho(S\subset X)=13.$
{\bf Challenge:} Verify the singularity 
assumption for some $X\in \cC_{42}$.
\end{exam}

\section{Decomposition of the diagonal}
Let $X$ be a smooth projective variety of dimension $n$ over $\bC$.
A {\em decomposition of the diagonal} of $X$ is a rational equivalence
in $\CH^n(X\times X)$
$$N\Delta_X \equiv N \{x \times X\} + Z,$$
where $N$ is a non-zero integer, $x\in X(\bC)$,
and $Z$ is supported on $X\times D$
for some subvariety $D\subsetneq X$. 
Let $N(X)$ denote the smallest positive integer $N$ occuring in a 
decomposition of the diagonal, which coincides with the greatest
common divisor of such integers.
If $N(X)=1$ we say $X$ admits an {\em integral decomposition 
of the diagonal}.

\begin{prop} \cite[Lem.~1.3]{ACTP} \cite[Lem.~4.6]{VoisinSDG}
\label{prop:DDChow} 
$X$ admits a decomposition of the diagonal if and only if 
$$A_0=\{ P \in \CH_0: \deg(P)=0\}$$
is universally $N$-torsion for some positive integer $N$, 
i.e., for each extension 
$F/\bC$ we have $NA_0(X_F)=0$. 
Moreover, $N(X)$ is the annihilator of the torsion.
\end{prop}

We sketch this for the convenience of the reader:
Basic properties of Chow groups give the equivalence of the
decomposition of $N\Delta_X$ with $NA_0(X_{\bC(X)})=0$.
Indeed, $A_0(X_{\bC(X)})$ is the inverse limit of 
$A_0(X \times U)$ for all open $U\subset X$.
Conversely, taking the basechange of a decomposition of the diagonal
to the extension $F$ gives that $A_0(X_F)$ is annihilated by $N$.

We recall situations where we have decompositions of the diagonal:

\subsection*{Rationally connected varieties}
Suppose $X$ is rationally connected and
choose $\beta \in H_2(X,\bZ)$ such that
the evaluation
$$M_{0,2}(X,\beta) \ra X \times X$$
is dominant. Fix an irreducible component $M$ of
$M_{0,2}(X,\beta)\times_{X\times X} \bC(X\times X)$.
Then $N(X)$ divides the index $\iota(M)$.
Indeed, each effective zero-cycle $Z\subset M$ 
corresponds to $|Z|$ conjugate rational curves joining 
generic $x_1,x_2 \in X$. Together these give $|Z|x_1=|Z|x_2$
in $\CH_0(X_{\bC(X)})$.
Thus we obtain a decomposition of the diagonal.
See \cite[Prop.~11]{CTRC} for more details.

\subsection*{Unirational varieties}
If $\rho: \bP^n \dashrightarrow X$ has degree $\varrho$ then
$N(X)|\varrho$. Thus $N(X)$ divides the greatest common divisor of
the degrees of unirational parametrizations of $X$.
A cubic hypersurface $X$ of dimension at least two admits
a degree two unirational parametrization (see Prop~\ref{prop:degreetwo}),
so $N(X)|2$.
We saw in \S\ref{sect:unirat} that many
classes of special cubic fourfolds admit odd degree unirational
parametrizations.
In these cases, we obtain {\em integral} decompositions of the diagonal.

\subsection*{Rational and stably rational varieties}
The case of rational varieties follows
from our analysis of unirational parametrizations. 
For the stably rational case, it suffices to observe that
$$A_0(Y)\simeq A_0(Y\times \bP^1)$$
and use the equivalence of Proposition~\ref{prop:DDChow}
(see \cite[Prop.~4.7]{VoisinSDG} for details).
Here we obtain an {\em integral} decomposition of the diagonal.

\

Remarkably, at least half of the special cubic fourfolds
admit integral decompositions of the diagonal:

\begin{theo} \cite[Th.~5.6]{VoisinJEMS} \label{theo:IDOD}
A special cubic fourfold of discriminant $d\equiv 2\pmod{4}$
admits an integral decomposition of the diagonal.
\end{theo}

This suggests the following question:

\begin{ques}
Do special cubic fourfolds of discriminant $d\equiv 2\pmod{4}$
always admit unirational parametrizations of odd degree?
Are they stably rational?
\end{ques}

Cubic fourfolds do satisfy a universal cohomological condition that follows from 
an integral decomposition of the diagonal: They admit universally
trivial unramified $H^3$. This was proved first for
cubic fourfolds containing a plane
($d=8$) using deep properties of quadratic forms \cite{ACTP},
then in general by Voisin \cite[Exam.~3.2]{VoisinIM}.

\begin{ques}
Is there a cubic fourfold $X$ with
$$K_d = H^{2,2}(X,\bZ), \quad d\equiv 0\pmod{4},$$
and admitting an integral decomposition of the diagonal?
A unirational parametriziation of odd degree?
\end{ques}

\bibliographystyle{alpha}
\bibliography{CubicLectures}

\begin{thebibliography}{ABBVA14}

\bibitem[ABBVA14]{ABBVA}
Asher Auel, Marcello Bernardara, Michele Bolognesi, and Anthony
  V{\'a}rilly-Alvarado.
\newblock Cubic fourfolds containing a plane and a quintic del {P}ezzo surface.
\newblock {\em Algebr. Geom.}, 1(2):181--193, 2014.

\bibitem[ABGvB13]{ABvB}
Asher Auel, Christian B{\"o}hning, and Hans-Christian Graf~von Bothmer.
\newblock The transcendental lattice of the sextic {F}ermat surface.
\newblock {\em Math. Res. Lett.}, 20(6):1017--1031, 2013.

\bibitem[ACTP13]{ACTP}
Asher Auel, Jean-Louis Colliot-Th\'el\`ene, and Raman Parimala.
\newblock Universal unramified cohomology of cubic fourfolds containing a
  plane.
\newblock In {\em Brauer Groups and Obstruction Problems:\ Moduli Spaces and
  Arithmetic}. 2013.
\newblock Proceedings of a workshop at the American Institute of Mathematics.

\bibitem[Add14]{AddMRL}
Nicolas Addington.
\newblock On two rationality conjectures for cubic fourfolds.
\newblock {\em Math. Res. Lett.}, 2014.
\newblock To appear, arXiv:1405.4902.

\bibitem[AK77]{AK}
Allen~B. Altman and Steven~L. Kleiman.
\newblock Foundations of the theory of {F}ano schemes.
\newblock {\em Compositio Math.}, 34(1):3--47, 1977.

\bibitem[AKMW02]{AKMW}
Dan Abramovich, Kalle Karu, Kenji Matsuki, and Jaros{\l}aw W{\l}odarczyk.
\newblock Torification and factorization of birational maps.
\newblock {\em J. Amer. Math. Soc.}, 15(3):531--572 (electronic), 2002.

\bibitem[AL15]{AL}
Nicolas Addington and Manfred Lehn.
\newblock On the symplectic eightfold associated to a {P}faffian cubic
  fourfold.
\newblock {\em J. Reine Angew. Math.}, 2015.

\bibitem[AT14]{AdTh}
Nicolas Addington and Richard Thomas.
\newblock Hodge theory and derived categories of cubic fourfolds.
\newblock {\em Duke Math. J.}, 163(10):1885--1927, 2014.

\bibitem[Bak10]{Baker}
H.~F. Baker.
\newblock {\em Principles of geometry. {V}olume 6. {I}ntroduction to the theory
  of algebraic surfaces and higher loci}.
\newblock Cambridge Library Collection. Cambridge University Press, Cambridge,
  2010.
\newblock Reprint of the 1933 original.

\bibitem[BD85]{BD}
Arnaud Beauville and Ron Donagi.
\newblock La vari\'et\'e des droites d'une hypersurface cubique de dimension
  {$4$}.
\newblock {\em C. R. Acad. Sci. Paris S\'er. I Math.}, 301(14):703--706, 1985.

\bibitem[Bea83]{Beau}
Arnaud Beauville.
\newblock Vari\'et\'es {K}\"ahleriennes dont la premi\`ere classe de {C}hern
  est nulle.
\newblock {\em J. Differential Geom.}, 18(4):755--782 (1984), 1983.

\bibitem[Bea00]{BeauDet}
Arnaud Beauville.
\newblock Determinantal hypersurfaces.
\newblock {\em Michigan Math. J.}, 48:39--64, 2000.
\newblock Dedicated to William Fulton on the occasion of his 60th birthday.

\bibitem[BHB06]{BrHB}
T.~D. Browning and D.~R. Heath-Brown.
\newblock The density of rational points on non-singular hypersurfaces. {II}.
\newblock {\em Proc. London Math. Soc. (3)}, 93(2):273--303, 2006.
\newblock With an appendix by J. M. Starr.

\bibitem[BHT15]{BMT}
Arend Bayer, Brendan Hassett, and Yuri Tschinkel.
\newblock Mori cones of holomorphic symplectic varieties of {K}3 type.
\newblock {\em Ann. Sci. \'Ec. Norm. Sup\'er. (4)}, 48(4):941--950, 2015.

\bibitem[BM14a]{BM2}
Arend Bayer and Emanuele Macr{\`{\i}}.
\newblock M{MP} for moduli of sheaves on {K}3s via wall-crossing: nef and
  movable cones, {L}agrangian fibrations.
\newblock {\em Invent. Math.}, 198(3):505--590, 2014.

\bibitem[BM14b]{BM1}
Arend Bayer and Emanuele Macr{\`{\i}}.
\newblock Projectivity and birational geometry of {B}ridgeland moduli spaces.
\newblock {\em J. Amer. Math. Soc.}, 27(3):707--752, 2014.

\bibitem[Bor14]{Bor}
Lev Borisov.
\newblock Class of the affine line is a zero divisor in the {G}rothendieck
  ring, 2014.
\newblock arXiv:1412.6194.

\bibitem[BRS15]{BRS}
Michele Bolognesi, Francesco Russo, and Giovanni Staglian\`o.
\newblock Some loci of rational cubic fourfolds, 2015.
\newblock arXiv:1504.05863.

\bibitem[CG72]{CG}
C.~Herbert Clemens and Phillip~A. Griffiths.
\newblock The intermediate {J}acobian of the cubic threefold.
\newblock {\em Ann. of Math. (2)}, 95:281--356, 1972.

\bibitem[CK89]{CrKa}
Bruce Crauder and Sheldon Katz.
\newblock Cremona transformations with smooth irreducible fundamental locus.
\newblock {\em Amer. J. Math.}, 111(2):289--307, 1989.

\bibitem[CT05]{CTRC}
Jean-Louis Colliot-Th{\'e}l{\`e}ne.
\newblock Un th\'eor\`eme de finitude pour le groupe de {C}how des
  z\'ero-cycles d'un groupe alg\'ebrique lin\'eaire sur un corps {$p$}-adique.
\newblock {\em Invent. Math.}, 159(3):589--606, 2005.

\bibitem[Dol05]{DolCr}
Igor~V. Dolgachev.
\newblock Luigi {C}remona and cubic surfaces.
\newblock In {\em Luigi {C}remona (1830--1903) ({I}talian)}, volume~36 of {\em
  Incontr. Studio}, pages 55--70. Istituto Lombardo di Scienze e Lettere,
  Milan, 2005.

\bibitem[Edg32]{edge}
W.~L. Edge.
\newblock The number of apparent double points of certain loci.
\newblock {\em Math. Proc. Cambridge Philos. Soc.}, 28:285--299, 1932.

\bibitem[Fan43]{Fano}
Gino Fano.
\newblock Sulle forme cubiche dello spazio a cinque dimensioni contenenti
  rigate razionali del {$4^\circ$} ordine.
\newblock {\em Comment. Math. Helv.}, 15:71--80, 1943.

\bibitem[Ful84]{Fulton}
William Fulton.
\newblock {\em Intersection theory}, volume~2 of {\em Ergebnisse der Mathematik
  und ihrer Grenzgebiete (3) [Results in Mathematics and Related Areas (3)]}.
\newblock Springer-Verlag, Berlin, 1984.

\bibitem[GS14]{GalShi}
Sergey Galkin and Evgeny Shinder.
\newblock The {F}ano variety of lines and rationality problem for a cubic
  hypersurface, 2014.
\newblock arXiv:1405.5154.

\bibitem[Has99]{HaJAG}
Brendan Hassett.
\newblock Some rational cubic fourfolds.
\newblock {\em J. Algebraic Geom.}, 8(1):103--114, 1999.

\bibitem[Has00]{HaCM}
Brendan Hassett.
\newblock Special cubic fourfolds.
\newblock {\em Compositio Math.}, 120(1):1--23, 2000.

\bibitem[HK07]{HuKl}
Klaus Hulek and Remke Kloosterman.
\newblock The {$L$}-series of a cubic fourfold.
\newblock {\em Manuscripta Math.}, 124(3):391--407, 2007.

\bibitem[HLOY03]{HLOY}
Shinobu Hosono, Bong~H. Lian, Keiji Oguiso, and Shing-Tung Yau.
\newblock Fourier-{M}ukai partners of a {$K3$} surface of {P}icard number one.
\newblock In {\em Vector bundles and representation theory ({C}olumbia, {MO},
  2002)}, volume 322 of {\em Contemp. Math.}, pages 43--55. Amer. Math. Soc.,
  Providence, RI, 2003.

\bibitem[HLOY04]{HLOY2}
Shinobu Hosono, Bong~H. Lian, Keiji Oguiso, and Shing-Tung Yau.
\newblock Autoequivalences of derived category of a {$K3$} surface and
  monodromy transformations.
\newblock {\em J. Algebraic Geom.}, 13(3):513--545, 2004.

\bibitem[HT01]{HTGAFA}
B.~Hassett and Y.~Tschinkel.
\newblock Rational curves on holomorphic symplectic fourfolds.
\newblock {\em Geom. Funct. Anal.}, 11(6):1201--1228, 2001.

\bibitem[HT09]{HTGAFA2}
Brendan Hassett and Yuri Tschinkel.
\newblock Moving and ample cones of holomorphic symplectic fourfolds.
\newblock {\em Geom. Funct. Anal.}, 19(4):1065--1080, 2009.

\bibitem[Huy99]{HuyInv}
Daniel Huybrechts.
\newblock Compact hyper-{K}\"ahler manifolds: basic results.
\newblock {\em Invent. Math.}, 135(1):63--113, 1999.

\bibitem[HVAV11]{HVAV}
Brendan Hassett, Anthony V{\'a}rilly-Alvarado, and Patrick Varilly.
\newblock Transcendental obstructions to weak approximation on general {K}3
  surfaces.
\newblock {\em Adv. Math.}, 228(3):1377--1404, 2011.

\bibitem[Kul08]{Kul}
Vik.~S. Kulikov.
\newblock A remark on the nonrationality of a generic cubic fourfold.
\newblock {\em Mat. Zametki}, 83(1):61--68, 2008.

\bibitem[Kuz10]{Kuz08}
Alexander Kuznetsov.
\newblock Derived categories of cubic fourfolds.
\newblock In {\em Cohomological and geometric approaches to rationality
  problems}, volume 282 of {\em Progr. Math.}, pages 219--243. Birkh\"auser
  Boston, Inc., Boston, MA, 2010.

\bibitem[Kuz15]{Kuz15}
Alexander Kuznetsov.
\newblock Derived categories view on rationality problems.
\newblock Lecture notes for the CIME--CIRM summer school, Levico Terme, 2015.
\newblock arXiv:1509.09115.

\bibitem[Laz10]{LazaAnnals}
Radu Laza.
\newblock The moduli space of cubic fourfolds via the period map.
\newblock {\em Ann. of Math. (2)}, 172(1):673--711, 2010.

\bibitem[LLSvS15]{LLSvS}
Christian Lehn, Manfred Lehn, Christoph Sorger, and Duco van Straten.
\newblock Twisted cubics on cubic fourfolds.
\newblock {\em J. Reine Angew. Math.}, 2015.

\bibitem[Loo09]{Loo}
Eduard Looijenga.
\newblock The period map for cubic fourfolds.
\newblock {\em Invent. Math.}, 177(1):213--233, 2009.

\bibitem[LZ13]{LZ}
Zhiyuan Li and Letao Zhang.
\newblock Modular forms and special cubic fourfolds.
\newblock {\em Adv. Math.}, 245:315--326, 2013.

\bibitem[Mar11]{MarkSurv}
Eyal Markman.
\newblock A survey of {T}orelli and monodromy results for
  holomorphic-symplectic varieties.
\newblock In {\em Complex and differential geometry}, volume~8 of {\em Springer
  Proc. Math.}, pages 257--322. Springer, Heidelberg, 2011.

\bibitem[May11]{Maya}
Evgeny Mayanskiy.
\newblock Intersection lattices of cubic fourfolds, 2011.
\newblock arXiv:1112.0806.

\bibitem[MFK94]{MFK}
D.~Mumford, J.~Fogarty, and F.~Kirwan.
\newblock {\em Geometric invariant theory}, volume~34 of {\em Ergebnisse der
  Mathematik und ihrer Grenzgebiete (2) [Results in Mathematics and Related
  Areas (2)]}.
\newblock Springer-Verlag, Berlin, third edition, 1994.

\bibitem[Mor40]{Morin}
Ugo Morin.
\newblock Sulla razionalit\`a dell'ipersuperficie cubica generale dello spazio
  lineare {$S_5$}.
\newblock {\em Rend. Sem. Mat. Univ. Padova}, 11:108--112, 1940.

\bibitem[MS12]{MacSte}
Emanuele Macr{\`{\i}} and Paolo Stellari.
\newblock Fano varieties of cubic fourfolds containing a plane.
\newblock {\em Math. Ann.}, 354(3):1147--1176, 2012.

\bibitem[Muk87]{MukaiTata}
S.~Mukai.
\newblock On the moduli space of bundles on {$K3$} surfaces. {I}.
\newblock In {\em Vector bundles on algebraic varieties ({B}ombay, 1984)},
  volume~11 of {\em Tata Inst. Fund. Res. Stud. Math.}, pages 341--413. Tata
  Inst. Fund. Res., Bombay, 1987.

\bibitem[Nik79]{Nik}
V.~V. Nikulin.
\newblock Integer symmetric bilinear forms and some of their geometric
  applications.
\newblock {\em Izv. Akad. Nauk SSSR Ser. Mat.}, 43(1):111--177, 238, 1979.

\bibitem[Nue15]{Nuer}
Howard Nuer.
\newblock Unirationality of moduli spaces of special cubic fourfolds and {$K3$}
  surfaces, 2015.
\newblock arXiv:1503.05256v1.

\bibitem[Orl97]{Orlov}
D.~O. Orlov.
\newblock Equivalences of derived categories and {$K3$} surfaces.
\newblock {\em J. Math. Sci. (New York)}, 84(5):1361--1381, 1997.
\newblock Algebraic geometry, 7.

\bibitem[Ran88]{RanThesis}
Kristian Ranestad.
\newblock {\em On smooth surfaces of degree 10 in the projective fourspace}.
\newblock PhD thesis, University of Oslo, 1988.

\bibitem[Ran91]{RanPaper}
Kristian Ranestad.
\newblock Surfaces of degree {$10$} in the projective fourspace.
\newblock In {\em Problems in the theory of surfaces and their classification
  ({C}ortona, 1988)}, Sympos. Math., XXXII, pages 271--307. Academic Press,
  London, 1991.

\bibitem[Siu81]{Siu}
Yum~Tong Siu.
\newblock A simple proof of the surjectivity of the period map of {$K3$}\
  surfaces.
\newblock {\em Manuscripta Math.}, 35(3):311--321, 1981.

\bibitem[ST01]{ST}
Paul Seidel and Richard Thomas.
\newblock Braid group actions on derived categories of coherent sheaves.
\newblock {\em Duke Math. J.}, 108(1):37--108, 2001.

\bibitem[Swa89]{Swan}
Richard~G. Swan.
\newblock Zero cycles on quadric hypersurfaces.
\newblock {\em Proc. Amer. Math. Soc.}, 107(1):43--46, 1989.

\bibitem[Tre84]{Tregub1}
S.~L. Tregub.
\newblock Three constructions of rationality of a cubic fourfold.
\newblock {\em Vestnik Moskov. Univ. Ser. I Mat. Mekh.}, (3):8--14, 1984.

\bibitem[Tre93]{Tregub2}
S.~L. Tregub.
\newblock Two remarks on four-dimensional cubics.
\newblock {\em Uspekhi Mat. Nauk}, 48(2(290)):201--202, 1993.

\bibitem[TVA15]{TVA}
Sho Tanimoto and Anthony V\'arilly-Alvarado.
\newblock Kodaira dimension of moduli of special cubic fourfolds, 2015.
\newblock arXiv:1509.01562v1.

\bibitem[vG05]{vanGeemen}
Bert van Geemen.
\newblock Some remarks on {B}rauer groups of {$K3$} surfaces.
\newblock {\em Adv. Math.}, 197(1):222--247, 2005.

\bibitem[Voi86]{Voisin86}
Claire Voisin.
\newblock Th\'eor\`eme de {T}orelli pour les cubiques de {${\bf P}^5$}.
\newblock {\em Invent. Math.}, 86(3):577--601, 1986.

\bibitem[Voi07]{VoisinJJM}
Claire Voisin.
\newblock Some aspects of the {H}odge conjecture.
\newblock {\em Jpn. J. Math.}, 2(2):261--296, 2007.

\bibitem[Voi13]{VoisinJAG}
Claire Voisin.
\newblock Abel-{J}acobi map, integral {H}odge classes and decomposition of the
  diagonal.
\newblock {\em J. Algebraic Geom.}, 22(1):141--174, 2013.

\bibitem[Voi14]{VoisinJEMS}
Claire Voisin.
\newblock On the universal {$\CH_0$} group of cubic hypersurfaces.
\newblock {\em J. Eur. Math. Soc. (JEMS)}, 2014.
\newblock To appear, arXiv:1407.7261v2.

\bibitem[Voi15a]{VoisinSDG}
Claire Voisin.
\newblock Stable birational invariants and the {L}\"uroth problem, 2015.
\newblock Preprint at \url{http://webusers.imj-prg.fr/~claire.voisin/}.

\bibitem[Voi15b]{VoisinIM}
Claire Voisin.
\newblock Unirational threefolds with no universal codimension {$2$} cycle.
\newblock {\em Invent. Math.}, 201(1):207--237, 2015.

\bibitem[W{\l}o03]{Wlo}
Jaros{\l}aw W{\l}odarczyk.
\newblock Toroidal varieties and the weak factorization theorem.
\newblock {\em Invent. Math.}, 154(2):223--331, 2003.

\bibitem[Yos00]{Yoshioka99}
K{\=o}ta Yoshioka.
\newblock Irreducibility of moduli spaces of vector bundles on {$K3$} surfaces,
  2000.
\newblock arXiv:9907001v2.

\bibitem[Yos01]{YoMA}
K{\=o}ta Yoshioka.
\newblock Moduli spaces of stable sheaves on abelian surfaces.
\newblock {\em Math. Ann.}, 321(4):817--884, 2001.

\end{thebibliography}

\end{document}